\documentclass[reqno,12pt,letterpaper]{amsart}
\usepackage{amsmath,amssymb,mathrsfs,graphicx,amsthm, mathtools}
\usepackage[dvipsnames]{xcolor}
\usepackage[colorlinks=true,linkcolor=Red,citecolor=Green]{hyperref}
\usepackage[T1]{fontenc}
\usepackage[utf8]{inputenc}
\usepackage{tikz-cd}

\newtheorem{proposition}{Proposition}[section]
\newtheorem{theorem}[proposition]{Theorem}
\newtheorem{lemma}[proposition]{Lemma}
\newtheorem{corollary}[proposition]{Corollary}
\theoremstyle{definition}
\newtheorem{definition}[proposition]{Definition}

\theoremstyle{remark}
\newtheorem{remark}[proposition]{Remark}
\theoremstyle{definition}

\numberwithin{equation}{section}

\newcommand{\R}{\mathbb{R}}

\usepackage[labelfont=up]{subcaption}

\title[Emergence of quantum dynamics from chaos]{Emergence of quantum dynamics from chaos: \\ The case of prequantum cat maps}
\author[J.~Echevarría Cuesta]{Javier Echevarría Cuesta}
\address{Centre de Mathématiques Laurent Schwartz, École Polytechnique, 91128 Palaiseau Cedex, France}
\email{javier.echevarria-cuesta@polytechnique.edu}

\begin{document}

\begin{abstract}
Faure and Tsujii recently proposed a new quantization theory for symplectic Anosov diffeomorphisms. It combines prequantization with the study of the Pollicott--Ruelle resonances of an associated transfer operator. We apply this framework to the hyperbolic symplectic automorphisms of the $2n$-dimensional torus, the so-called cat maps. Our main result gives an explicit relation between the resonances of the prequantum transfer operator and the eigenvalues of the standard quantum cat maps, generalizing the case $n=1$ previously treated by Faure.
\end{abstract}
\maketitle

\section{Introduction}

Early in the inception of quantum mechanics, it became clear that the differences with classical mechanics were too drastic, both in the predictions and the mathematical formalism, for them not to warrant a deeper explanation. We now begin to understand, through Bohr's correspondence principle, how the classical world emerges as a limiting case of its quantum counterpart. Celebrated results in this direction include Egorov's theorem and the WKB approximation. In recent years, work by Faure and Tsujii has hinted at a potential relationship in the other direction: they have shown that quantum dynamics can emerge from the long-term behaviour of both  discrete \cite{ faure07, faure15} and continuous \cite{faure24} chaotic classical systems.

Their observations relate to the problem of quantization, that is, going from a classical system to its quantum counterpart. Mathematically, quantization is subtle because it involves choices, making it unclear which procedures are more canonical. Until recently, two general approaches existed: geometric quantization and the algebraic theory of deformation quantization. These are dual to each other, much like the Heisenberg and Schrödinger pictures in quantum mechanics.

It turns out that the first step of geometric quantization, known as prequantization, can be compelling in its own right. Roughly speaking, prequantum dynamics is the same as classical dynamics but with the introduction of complex phases. It should feel natural to introduce complex phases because these govern interference effects, a key characteristic of wave and quantum dynamics. The evolution of these phases over time is dictated by the classical trajectories. More precisely, prequantization is equivalent to picking a contact $\text{U}(1)$-extension of the classical dynamics.

The Hilbert space of wave functions resulting from prequantization is too large to be the final product of a proper quantization procedure. Heuristically, this is because there is no uncertainty principle. Geometric quantization solves this by choosing a complex structure on phase space, that is, a complex polarization, but this is an arbitrary choice and restricts the class of functions that can be quantized.

Faure and Tsujii's recent work shows that, for chaotic classical systems, there may be a better way to achieve a full quantization theory. The missing ingredient appears to be the long-term dynamics of the prequantum transfer operator. The method is worked out for symplectic Anosov diffeomorphisms in \cite{faure15} and contact Anosov flows in \cite{faure24}. In this paper, we apply it to hyperbolic symplectic automorphisms of the $2n$-dimensional torus, the so-called cat maps. In Theorem \ref{theorem:main-result}, we show that the Pollicott--Ruelle resonances of the prequantum transfer operator are related to the usual Weyl spectrum of the quantum evolution operators. We then show how quantum behaviour appears dynamically in the correlation functions.

This is akin to the results in \cite{faure07} for $n=1$, but with several differences. First, \cite{faure07} only considers cat maps that are the time-$1$ flow of a quadratic Hamiltonian on $\mathbb{R}^2$. This is not too restrictive for $n=1$ because $\text{Sp}(2, \mathbb{R})=\text{SL}(2, \mathbb{R})$, so each matrix can be written as $\pm e^X$ for some $X\in \mathfrak{sl}(2,\mathbb{R})$. In higher dimensions, the assumption is much stronger. Second, we clarify the statements on the uniqueness of the objects, the need for parity conditions on the matrix coefficients, and related points that were previously ignored or left unexplained. Finally, rather than using the complex line bundle terminology, we adopt the language of principal $\text{U}(1)$-bundles to harmonize the concepts and notation with the more recent work \cite{faure15}.

Quantum cat maps are toy models often used to study quantum chaos \cite{zelditch87, degli-esposti93, bouzouina96, faure03}. A natural question is whether the prequantum approach can shed light on the semiclassical behaviour of the eigenstates of the quantum dynamics, which is the central problem in the study of quantum cat maps. Accordingly, another objective of this paper is to provide a concrete realization of all prequantum cat maps. Part of this was done in the earlier work \cite{zelditch97} on Toeplitz quantization, but only a special subgroup was explicitly prequantized.

\subsection*{Acknowledgements} I would like to thank Frédéric Faure for introducing me to this topic and always being available for discussion. This work was supported by a Sophie Germain Scholarship from the Fondation Mathématique Jacques Hadamard. 

\section{Statement of the results}\label{section:statement-of-the-results}

Let $\mathbb{T}^{2n}\coloneqq\mathbb{R}^{2n}/\mathbb{Z}^{2n}$ be the $2n$-dimensional torus. We use coordinates $x=(q,p)$ to represent a point on either $\mathbb{T}^{2n}$ or its universal cover $\mathbb{R}^{2n}$. We endow $\mathbb{T}^{2n}$ with the standard symplectic $2$-form $\omega \coloneq\sum_{j=1}^n dq_j\wedge dp_j$.

A matrix $M\in \text{Sp}(2n, \mathbb{Z})$ descends to a symplectic automorphism of $\mathbb{T}^{2n}$. If $M$ is hyperbolic, that is, if it has no eigenvalues on the unit circle $\mathbb{S}^1$, the resulting automorphism is Anosov (uniformly hyperbolic). Following the physics literature, we call it a \emph{classical cat map}. These maps are interesting because they exhibit strong chaotic properties such as ergodicity and mixing. They serve as discrete proxies for the geodesic flow on a negatively curved Riemannian manifold. We are interested in studying their prequantization.
\begin{definition}\label{def:desiderata-prequantum}
A \emph{prequantum bundle} $(P, \pi, \alpha)$ over the closed symplectic manifold $(\mathbb{T}^{2n}, \omega)$ is a smooth principal $\text{U}(1)$-bundle $\pi : P\to \mathbb{T}^{2n}$ equipped with a principal $\text{U}(1)$-connection $\alpha\in \mathcal{C}^\infty(P,T^*P\otimes i\mathbb{R})$ such that the curvature $2$-form $\Omega=d\alpha$ satisfies
\begin{equation}\label{eq:curvature-original}
\Omega=2\pi i \cdot \pi^*\omega.
\end{equation}
Given a hyperbolic matrix  $M\in \textup{Sp}(2n,\mathbb{Z})$, a $\text{U}(1)$-equivariant lift $\widetilde{M}:P\to P$ preserving $\alpha$ is called a \emph{prequantum cat map}.
\end{definition}
We will verify that such objects exist and discuss to what degree they are unique.
\begin{remark}
We earlier characterized prequantization as a contact $\text{U}(1)$-extension of  the Hamiltonian dynamics. Indeed, if we define the $1$-form $\alpha'\coloneq\frac{1}{2\pi i} \alpha$, then
\begin{equation*}
dV_{P}\coloneq\frac{1}{n!}\alpha'\wedge (d\alpha')^{\wedge n}=\alpha'\wedge \pi^*dV_{\mathbb{T}^{2n}}
\end{equation*}
is a non-degenerate volume form on $P$, making $\alpha'$ a contact form preserved by $\widetilde{M}$.
\end{remark}

Prequantum dynamics are best described by a transfer operator.
\begin{definition}
The \emph{prequantum transfer operator} $F: \mathcal{C}^\infty(P)\to \mathcal{C}^\infty(P)$ associated to a prequantum cat map $\widetilde{M}:P\to P$ is defined as
\begin{equation*}
Fu\coloneq u\circ \widetilde{M}^{-1}, \quad u\in \mathcal{C}^\infty(P).
\end{equation*}
\end{definition}
This type of operator is pervasive in the study of dynamical systems because, once extended by duality to the space of distributions (generalized functions) $\mathcal{D}'(P)$, it corresponds to the pushforward evolution of probability distributions by the dynamics. To avoid dealing with chaotic individual trajectories, it is common to consider the behaviour of \emph{correlation functions}
\begin{equation*}
C_{u,v}(t)\coloneq\int_{P}  ( u\circ \widetilde{M}^{-t})\cdot \overline{v} \, dV_P=\langle \widehat{M}^t u, v\rangle_{L^2(P)},\quad u,v \in \mathcal{C}^\infty(P),
\end{equation*}
as $t\to \infty$. They measure the loss of memory or asymptotic independence.



One nice consequence of the $\text{U}(1)$-equivariance of $\widetilde{M}: P\to P$ is that the transfer operator $F$ commutes with the action of $\text{U}(1)$ on $\mathcal{C}^\infty(P)$. As a result, it respects a natural decomposition into Fourier modes with respect to the action of $\text{U}(1)$.
\begin{definition}
For $N\in \mathbb{Z}$, the subspace of functions in the \emph{$N$th Fourier mode} is
\begin{equation*}
\mathcal{C}^\infty_N(P)\coloneq\{ u\in \mathcal{C}^\infty(P)\mid u(e^{2\pi is} p)=e^{2\pi iNs} u(p)\text{ for all }p\in P, \, s\in \mathbb{R}\}.
\end{equation*}
We denote the restriction of $F$ to $\mathcal{C}^\infty_N(P)$ by $F_N$.
\end{definition}
To understand the long-time behaviour of the correlation functions, it suffices to study the correlation functions of each transfer operator $F_N$. Indeed, we have
\begin{equation*}
C_{u,v}(t)=\langle F^t u, v \rangle_{L^2(P)}=\sum_{N\in \mathbb{Z}}\langle F^t_N u_N, v_N \rangle_{L^2(P)},
\end{equation*}
where $u_N, v_N\in \mathcal{C}^\infty_N(P)$ are the $N$th Fourier components of $u,v\in \mathcal{C}^\infty(P)$.

\begin{remark}
$(i)$ Complex conjugation commutes with $F$ and maps $\mathcal{C}^\infty_N(P)$ to $\mathcal{C}^\infty_{-N}(P)$. It is thus enough to study the case $N\geq 0$. 

\noindent $(ii)$ When $N=0$, we are effectively dealing with smooth functions on the torus and classical cat maps. Using Fourier series, we can show (see \cite[p. 946]{brini01}) that we have super-exponential decay of correlations, that is, for any $\rho\in (0,1)$,
\begin{equation*}
C_{u, v}(t)=\int_{P}u\,dV_P  \int_{P}\overline{v}\,dV_P+\mathcal{O}_{u,v}(\rho^t).
\end{equation*}
We therefore restrict our attention to $N\geq 1$.


\noindent $(iii)$ The subspace $\mathcal{C}_N^\infty(P)$ can be identified with the space of smooth sections of an associated Hermitian complex line bundle $L^{\otimes N}$ over $\mathbb{T}^{2n}$ (that is, the $N$th tensor power of a line bundle $L\to \mathbb{T}^{2n}$) equipped with a covariant derivative. This is the point of view adopted in \cite{faure07} and most references on geometric quantization.


\end{remark}

Following \cite{ruelle86}, we take the \emph{power spectrum}, which for us is the Fourier transform $\widehat{C}_{u,v}$ of $C_{u,v}$ restricted to $t> 0$, with $u, v\in\mathcal{C}^\infty_N(P)$. By \cite{faure15}, $\widehat{C}_{u,v}$ is analytic in  $\{\lambda\in \mathbb{C}\mid \text{Im}(\lambda) < 0\}$ and admits a meromorphic extension to $\mathbb{C}$ with finite-rank poles, whose location and rank are independent of $u$ and $v$. The \emph{Pollicott--Ruelle resonances} are the complex numbers $e^{i\lambda}$ such that $\lambda\in \mathbb{C}$ is a pole of this meromorphic extension. They govern the long-time behaviour of $C_{u,v}(t)$.

The microlocal method of Faure--Tsujii in \cite{faure15} alternatively defines Pollicott--Ruelle resonances for each operator $F_N$ (extended by duality to $\mathcal{D}'_N(P)$, the dual space of $\mathcal{C}_N^\infty(P)$) as the points $\lambda \in \mathbb{C}$ for which the bounded operator
\begin{equation*}
\lambda I -F_N: H^r_N(P)\to H^r_N(P), \quad c^{r}<|\lambda|,
\end{equation*}
is not invertible. Here, each Hilbert space $H^r_N(P)$ is an \emph{anisotropic Sobolev space}
$$\mathcal{C}^\infty_N(P)\subset H^r_N(P)\subset \mathcal{D}'_N(P),$$ and $c\in (0,1)$ is a fixed constant independent of $r>0$ (see \cite[Theorem 1.3.1]{faure15}). Resonances and their associated eigenspaces are shown to be independent of the parameter $r$. The relation with our definition is given by the formula
\begin{equation*}
\widehat{C}_{u,v}(\lambda)\coloneq\sum_{t=1}^\infty e^{-i\lambda t} C_{u,v}(t)=\sum_{t=1}^\infty e^{-i\lambda t}\langle F^t_N u, v\rangle_{L^2(P)}=\langle (e^{i\lambda} I -F_N)^{-1}u, v\rangle_{L^2(P)},
\end{equation*}
valid for $\text{Im}(\lambda)<0$ and $u,v\in \mathcal{C}^\infty_N(P)$.

The first result in this paper is a description of these resonances. It is rare to find systems for which resonances can be explicitly computed. Moreover, Theorem \ref{theorem:main-result} exposes a link between the resonances of the transfer operators $F_N$ and the spectrum $\sigma(U_{N, \theta})$ of the usual quantum cat maps $U_{N, \theta}$ obtained through Weyl quantization (parametrized by $N\in \mathbb{N}^*$ and $\theta\in \mathbb{T}^{2n}$) and acting on the $N^n$-dimensional Hilbert spaces $\mathcal{H}_{N, \theta}$ (see Section \ref{subsection:quantum-cat-maps} for their definition).

To cleanly state the relationship, we impose the condition $\varphi_M=0$, where $\varphi_M$ is an element of $\{0,1\}^{2n}$ uniquely defined for each cat map $M$ via Lemma \ref{lemma:dyat}.








\begin{theorem}\label{theorem:main-result}
Let $M\in \textup{Sp}(2n,\mathbb{Z})$ be hyperbolic. We can find $E\in \textup{GL}(n,\mathbb{R})$ satisfying $\Vert E^{-1}\Vert <1$ and $|\det E|>1$ such that  $M$ is symplectically conjugate to
\begin{equation*}
\left(\begin{matrix}
E & 0\\
0 & (E^T)^{-1}\\
\end{matrix}\right).
\end{equation*}
Suppose that $\varphi_M=0$. Then, for any prequantum transfer operator $F$, the Pollicott--Ruelle resonances of $F_N$ with $N\in \mathbb{N}^*$ are, up to a global phase, given by
\begin{equation*}
|\det E|^{-1/2} \cdot \{\lambda_{j}\}_{j\in \mathbb{N}} \cdot \sigma(U_{N,0}),
\end{equation*}
where $\{\lambda_j\}_{j\in \mathbb{N}}$ are the eigenvalues of the operator $u \mapsto u\circ E^{-1}$ acting on $\mathcal{C}^\infty(\mathbb{R}^n)$.
\end{theorem}

\begin{remark}
$(i)$  The eigenvalues of $u\mapsto u\circ E^{-1}$ are contained in the annuli
\begin{equation}\label{eq:band-structure-2}
\left\{z\in \mathbb{C}\mid  \Vert E\Vert^{-k}\leq |z|\leq  \Vert E^{-1}\Vert^k\right\}
\end{equation}
indexed by $k\in \mathbb{N}$ and corresponding to the restriction to homogeneous polynomials on $\mathbb{R}^n$ of degree $k$. While these annuli can intersect each other, the outermost one is always isolated because  $\Vert E^{-1}\Vert < 1$.

\noindent $(ii)$ We have $\sigma(U_{N, \theta})\subset \mathbb{S}^1$ and  $|\sigma(U_{N, \theta})|=N^n$ counting multiplicities. Out of all the operators $U_{N, \theta}$, only $U_{N,0}$ appears here because of the condition $\varphi_M=0$.

\noindent $(iii)$ The case $n=1$ is treated in \cite{faure07}. However, as discussed above, the paper deals with a restricted class of cat maps. Moreover, instead of assuming $\varphi_M=0$, it restricts the result to even $N$.



\end{remark}

\begin{figure}[htp]
\centering
\begin{subfigure}{0.49\textwidth}
\centering
\includegraphics[width=\textwidth]{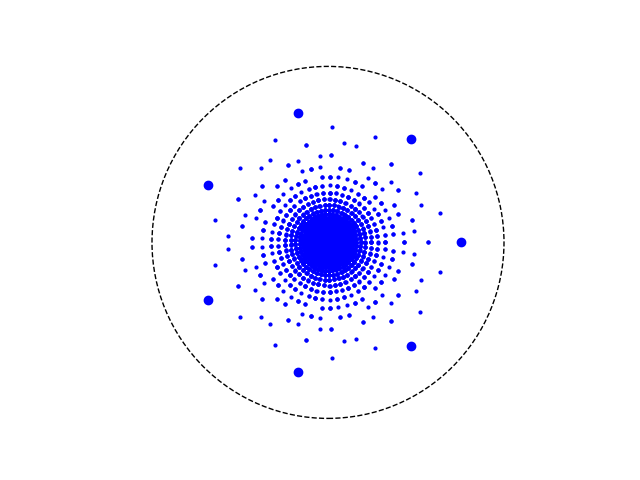}
\caption{prequantum resonances}
\label{figure:spectra-a}
\end{subfigure}
\hfill
\begin{subfigure}{0.49\textwidth}
\centering
\includegraphics[width=\textwidth]{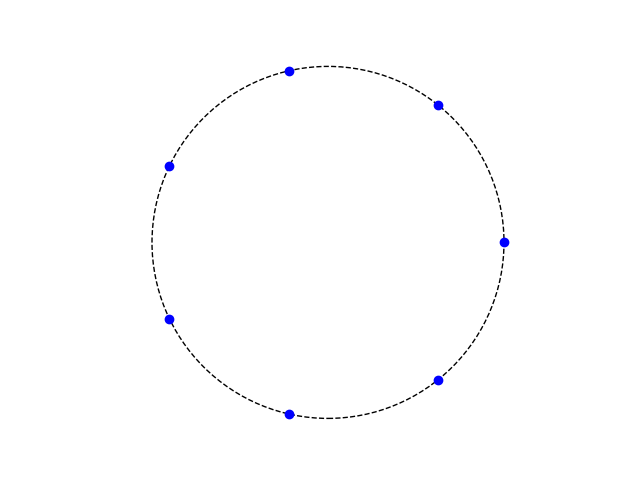}
\caption{quantum eigenvalues}
\label{figure:spectra-b}
\end{subfigure}
\caption{Spectra for a cat map $M\in \mathrm{Sp}(6,\mathbb{Z})$, with $N=2$.}
\end{figure}

The main motivation for studying resonances is that correlation functions can be expressed as  asymptotic expansions over them, up to exponentially small error (see \cite[Theorem 1.6.3]{faure15}). Theorem \ref{theorem:second-result} makes precise what we mean by the emergence of quantum dynamics: it shows that $C_{u, v}$ eventually behaves like quantum correlation functions, that is, quantum elements of the quantum propagators.

The argument relies on the fact that, as $t\to \infty$, the external prequantum resonances on the circle of radius $|\det E|^{-1/2}$ dominate the dynamics. See Figure \ref{figure:spectra-a}. This is why the quantization scheme of Faure--Tsujii \cite[Definition 1.3.6]{faure15} consists in spectrally projecting the transfer operator onto the outermost annulus of the Pollicott--Ruelle spectrum. We do this explicitly. As we will see in Corollary \ref{corollary:observables-preserved}, there are explicit isomorphisms
$$W_{N, \theta}:\mathcal{H}_{N, \theta}\otimes \mathcal{S}(\mathbb{R}^n) \to \mathcal{C}^\infty_N(P).$$
To obtain quantization maps $\mathcal{C}_N^\infty(P)\to \mathcal{H}_{N, \theta}$, mapping observables to quantum states,  we compose $W_{N,\theta}^{-1}$ with a projector $\mathcal{H}_{N, \theta}\otimes \mathcal{S}(\mathbb{R}^n)\to \mathcal{H}_{N, \theta}$ such as
\begin{equation}\label{eq:projectors}
(I\otimes \delta_0)(\nu \otimes f)= \nu f(0)\quad \text{ or } \quad (I\otimes \, dx)(\nu \otimes f) = \nu \int_{\mathbb{R}^n}f(x)\, dx.
\end{equation}

\begin{theorem}\label{theorem:second-result} Suppose that $\varphi_M=0$, pick $E\in \textup{GL}(n,\mathbb{R})$ as in Theorem \ref{theorem:main-result} and fix $N\in \mathbb{N}^*$. For any $u, v\in \mathcal{C}^\infty_N(P)$, define $\hat{u},\hat{v}\in\mathcal{H}_{N, 0}$ by
$$\hat{u}\coloneq(I\otimes \delta_0)(W_{N,0}^{-1}u)\quad \text{ and } \quad \hat{v}\coloneq(I\otimes dx)(W_{N,0}^{-1}v).$$
Then, for any $\rho\in (\Vert E^{-1}\Vert, 1)$, we have
\begin{equation*}
C_{u,v}(t)=\langle U_{N, 0}^t \hat{u}, \hat{v}\rangle_{\mathcal{H}_{N,0}}|\det E|^{-t/2}\left(1+\mathcal{O}_{u,v}(\rho^t)\right)\quad \text{as } t\to \infty.
\end{equation*}
\end{theorem}






\section{Prequantization of classical cat maps}


In this section, we construct the objects introduced in Definition \ref{def:desiderata-prequantum} and discuss their varying degrees of uniqueness.

\subsection{Construction of the prequantum bundle}

A key ingredient in the construction of the prequantum bundle is the reduced $(2n+1)$-dimensional Heisenberg group $\mathbb{H}^{\text{red}}_n$. As a reference, the Heisenberg group and its representations are thoroughly introduced in \cite[Chapter 1]{folland89}.   We recall that the group $\mathbb{H}^{\text{red}}_n$ consists of the set $\mathbb{R}^{2n}\times \mathbb{S}^1$ endowed with the group law
\begin{equation*}
\left(x,e^{2\pi i s}\right)\cdot (x', e^{2\pi i s'})\coloneq\left(x+x',e^{2\pi i\left(s+s'+\frac{1}{2}\sigma(x,x')\right)}\right),
\end{equation*}
where $\sigma:\mathbb{R}^{2n}\times \mathbb{R}^{2n}\to \mathbb{R}$ is the usual symplectic bilinear form on $\mathbb{R}^{2n}$ defined by
\begin{equation*}
\sigma(x, x')\coloneq \langle p, q'\rangle-\langle p', q\rangle.
\end{equation*}
This group first makes an appearance because it provides the canvas on which to build a non-trivial principal $\text{U}(1)$-bundle over $\mathbb{T}^{2n}$. Concretely, define the lattice $\Gamma\subset \mathbb{H}^{\text{red}}_n$ as the image of the group homomorphism $\mathbb{Z}^{2n}\to  \mathbb{H}^{\text{red}}_n$ given by
\begin{equation*}
w\mapsto (w, e^{\pi i Q(w)}),
\end{equation*}
where the quadratic form $Q$ on $\mathbb{R}^{2n}$ is
\begin{equation}\label{eq:quadratic-form}
Q(x)\coloneq \langle q, p\rangle.
\end{equation}
Then, we obtain a principal $\text{U}(1)$-bundle by considering $\pi :\Gamma\setminus\mathbb{H}^{\text{red}}_{n} \to \mathbb{T}^{2n}$. The map $\pi$ is nothing but the projection onto the first factor. The center of $\mathbb{H}^{\text{red}}_{n}$ is the $\mathbb{S}^1$ factor, and its $\text{U}(1)$-action generates the fibres of the principal $\text{U}(1)$-bundle.

In the following statement, two prequantum bundles $(P, \pi, \alpha)$ and $(P', \pi', \alpha')$ over $(\mathbb{T}^{2n},\omega)$ are considered to be equivalent if there exists a $\text{U}(1)$-equivariant bundle isomorphism $\varphi:P\to P'$ such that $\varphi^*\alpha'=\alpha$.

\begin{proposition}\label{proposition:classification}
Any prequantum bundle over $(\mathbb{T}^{2n},\omega)$ is equivalent to exactly one of $(\Gamma\setminus \mathbb{H}^{\text{red}}_{n}, \pi, \alpha_\kappa)$, $\kappa\in [0,1)^{2n}$, where $\alpha_\kappa$ is the principal $\textup{U}(1)$-connection given by
\begin{equation*}
\alpha_\kappa\coloneq 2\pi i\left( ds+\frac{1}{2} \sum_{j=1}^n(q_j\, dp_j-p_j\, dq_j)+\sum_{j=1}^n (\kappa_{j}\, dp_j-\kappa_{j+n}\, dq_j)\right).
\end{equation*}

\end{proposition}
\begin{proof}
Each $\alpha_\kappa$ is well defined on the quotient $\Gamma\setminus \mathbb{H}^{\text{red}}_{n}$ because it is left-invariant under the action of $\Gamma$. This becomes even easier to check if we more succinctly write
\begin{equation*}
\alpha_\kappa =2\pi i \left(ds+\frac{1}{2}\sigma( dx, x)+\sigma(dx, \kappa)\right).
\end{equation*}
Note that the normalization $\alpha_\kappa(\frac{\partial}{\partial s})=2\pi i$ is respected, and the $\text{U}(1)$-action leaves  $\alpha_\kappa$ invariant, so $\alpha_\kappa$ is indeed a principal $\text{U}(1)$-connection. The condition
\begin{equation*}
d\alpha_\kappa=2\pi i \cdot \pi^*\omega
\end{equation*}
is also satisfied, as desired. This confirms that we are dealing with an explicit family of prequantum bundles over $(\mathbb{T}^{2n}, \omega)$ parametrized by $\kappa\in [0, 1)^{2n}$.

All other choices of prequantum bundles are equivalent to $( \Gamma\setminus \mathbb{H}^{\text{red}}_{n}, \pi)$ up to bundle isomorphism because the first Chern class is always $[\omega]\in H^2(\mathbb{T}^{2n};\mathbb{Z})$. We can hence fix the principal $\text{U}(1)$-bundle to be $(\Gamma\setminus \mathbb{H}^{\text{red}}_{n}, \pi)$ and consider principal $\text{U}(1)$-connections up to gauge transformation.

The gauge transformations of $(\Gamma\setminus \mathbb{H}^{\text{red}}_{n}, \pi)$ are in one-to-one correspondence with $\mathcal{C}^\infty(\mathbb{T}^{2n}, \text{U}(1))$. Given $f$ in the latter space, the corresponding gauge transformation is $g_f(x, e^{2\pi i s})\coloneq (x,  f(x)e^{2\pi i s})$. It is well defined on the quotient it commutes with the action of $\Gamma$. For a principal $\text{U}(1)$-connection $\alpha$, we have $g_f^*\alpha= \alpha+\pi^*(f^{-1}\, df)$.

Let $\alpha$ be a principal $\text{U}(1)$-connection on $\Gamma\setminus \mathbb{H}^{\text{red}}_{n}$. The difference $\alpha-\alpha_0$ is invariant by the action of $\text{U}(1)$, so we may write $\alpha-\alpha_0 = 2\pi i \cdot \pi^*\beta$ for some $1$-form $\beta$ on $\mathbb{T}^{2n}$. If $\alpha$ satisfies the curvature condition \eqref{eq:curvature-original}, then $\beta$ must be closed. As a result, using the Hodge decomposition, we can write $\beta = dr+ h$, where $r\in \mathcal{C}^\infty(\mathbb{T}^{2n}, \mathbb{R})$ and $h$ is a unique harmonic $1$-form on $\mathbb{T}^{2n}$. Now using the gauge corresponding to $e^{-2\pi i r}$, we can get rid of $dr$ and assume that $\beta=h$. Harmonic forms on $\mathbb{T}^{2n}$ have the form $h = \langle y , dx\rangle$ with $y\in \mathbb{R}^{2n}$. There is a unique $w\in \mathbb{Z}^{2n}$ such that $y-w\in [0,1)^{2n}$, so we can use the gauge corresponding to $x\mapsto e^{-2\pi i \langle x, w\rangle}$ to bring $h$ to another harmonic $1$-form with coefficients in $[0,1)^{2n}$. By combining two gauges, we can thus find a gauge $g$ such that $g^*\alpha=\alpha_{\kappa}$ for some $\kappa\in [0,1)^{2n}$. This completes the proof since tracing through the uniqueness statements shows that $\alpha_\kappa$ is equivalent to $\alpha_{\kappa'}$ up to gauge transformation if and only if $\kappa=\kappa'$.
\end{proof}

\subsection{Construction of the prequantum cat maps}

It will be convenient to express a given matrix $M\in \text{Sp}(2n, \mathbb{Z})$ in block form
\begin{equation*}
M=\left(\begin{matrix}
A & B \\
C & D\\
\end{matrix}\right).
\end{equation*}
We recall that $M\in \text{Sp}(2n, \mathbb{Z})$ if and only if $A^T D-C^T B=I$, $A^T C=C^T A$ and $B^T D=D^T B$. We will need the following lemma.

\begin{lemma}\label{lemma:dyat} For each $M\in \textup{Sp}(2n, \mathbb{Z})$, there is a unique $\varphi_M\in \{0,1\}^{2n}$ such that
\begin{equation}\label{eq:dyat}
Q(M^{-1}w)-Q(w)=\sigma(\varphi_M, w) \mod 2\mathbb{Z} \quad \text{ for all } w\in\mathbb{Z}^{2n},
\end{equation}
with the quadratic form $Q$ defined in equation \eqref{eq:quadratic-form}.
\end{lemma}

\begin{proof}
Let us define a map $\mathbb{Z}^{2n}\to \mathbb{Z}/2\mathbb{Z}$ by
\begin{equation*}
w\mapsto Q(M^{-1}w)-Q(w) \mod 2\mathbb{Z}.
\end{equation*}
It is easy to check that this is a group homomorphism by using the relation
\begin{equation*}
Q(w+w')=Q(w)+Q(w')+\sigma(w,w') \mod 2\mathbb{Z} \quad\text{ for all } w,w'\in\mathbb{Z}^{2n},
\end{equation*}
and the fact that $M\in \text{Sp}(2n,\mathbb{Z})$. The existence and uniqueness of $\varphi_M$ follow.
\end{proof}

\begin{remark}\label{remark:varphi-m}
\noindent $i)$
We note that the map $M\to \varphi_M$ satisfies
\begin{equation*}
\varphi_{M^{-1}}=M^{-1}\varphi_M\mod (2\mathbb{Z})^{2n} \quad \text{ and }\quad \varphi_{MM'}=\varphi_M+M\varphi_{M'}\mod (2\mathbb{Z})^{2n}.
\end{equation*}

\noindent $ii)$ We also notice that
\begin{equation*}
Q(M^{-1}w)-Q(w)=\langle CD^T m, m\rangle+ \langle AB^T n, n\rangle \mod 2\mathbb{Z} \quad \text{for all } w=(m, n)\in \mathbb{Z}^{2n}.
\end{equation*}
Therefore, if the integer matrices $CD^T $ and $ AB^T$ have even  entries, we get $\varphi_M=0$.

\noindent $iii)$ In the case $n=1$, we may explicitly write
\begin{equation*}
\varphi_M=(CD, AB)\mod (2\mathbb{Z})^2.
\end{equation*}
Motivated by this, we say that cat maps with $\varphi_M=0$ are in chequerboard form.
\end{remark}

We proceed to give a concrete realization of the prequantum cat maps.

\begin{proposition}\label{prop:explicit-construction-map}
Let $M\in \textup{Sp}(2n, \mathbb{Z})$ be hyperbolic and define $\tilde{\kappa}\in \mathbb{R}^{2n}$ by
\begin{equation}\label{eq:prequantum-condition}
\tilde{\kappa} \coloneq \frac{1}{2}(I-M)^{-1}\varphi_{M}.
\end{equation}
The map $\widetilde{M}: \mathbb{H}^{\textup{red}}_{n}\to  \mathbb{H}^{\textup{red}}_{n}$ given by
\begin{equation*}
\widetilde{M}(x, e^{2\pi is})\coloneq\left(Mx, e^{2\pi i\left(s+\frac{1}{2}\sigma(\varphi_{M},Mx)\right)}\right)
\end{equation*}
descends to a map on the quotient $\Gamma\setminus \mathbb{H}^{\textup{red}}_{n}$ which is the unique (up to a global phase) $\textup{U}(1)$-equivariant lift of $M:\mathbb{T}^{2n}\to \mathbb{T}^{2n}$ preserving the principal $\textup{U}(1)$-connection
\begin{equation*}
\alpha_{\tilde{\kappa}} \coloneq 2\pi i \left(ds+\frac{1}{2}\sigma(dx,x)+\sigma(dx,\tilde{\kappa})\right).
\end{equation*}
Moreover, there is no $\textup{U}(1)$-equivariant lift of $M$ preserving $\alpha_{\tilde{\kappa}'}$ if $\tilde{\kappa}'\neq \tilde{\kappa}$.
\end{proposition}

\begin{proof}
The map $\widetilde{M}$ is well defined on the quotient $\Gamma\setminus  \mathbb{H}^{\textup{red}}_{n}$ because it preserves the subgroup $\Gamma$. Indeed, for any $w\in \mathbb{Z}^{2n}$, we have
\begin{equation*}
\left(Mw, e^{\pi i( Q(w)+\sigma(\varphi_{M},Mw))}\right)=\left(Mw, e^{\pi i Q(Mw)}\right)
\end{equation*}
thanks to relation \eqref{eq:dyat}. The fact that $\widetilde{M}$ is a $\text{U}(1)$-equivariant lift of $M$ is then clear. Since $M$ is hyperbolic, the definition of $\tilde{\kappa} $ makes sense. We claim that $\widetilde{M}$ preserves the principal $\text{U}(1)$-connection $\alpha_{\tilde{\kappa}}$. To see why this is true, note that writing $(x', e^{2\pi is'})=\widetilde{M}(x,e^{2\pi is})$ is equivalent to having $x'=Mx$ and $s'=s+\frac{1}{2}\sigma(\varphi_{M}, Mx)$. As a result, by using equation \eqref{eq:prequantum-condition},  we obtain
\begin{equation}\label{eq:preserve-connection-calc}
\begin{split}
\widetilde{M}^*\alpha_{\tilde{\kappa}} &= 2\pi i \left(ds'+\frac{1}{2}\sigma(dx',x')+\sigma(dx',\tilde{\kappa})\right)\\
&=2\pi i \left(ds+\frac{1}{2}\sigma(M^{-1}\varphi_{M}, dx)+\frac{1}{2}\sigma(dx,x)+\sigma(dx,M^{-1}\tilde{\kappa})\right)\\
&=2\pi i \left(ds+\frac{1}{2}\sigma(dx,x)+\sigma(dx,\tilde{\kappa})\right).
\end{split}
\end{equation}

A $\text{U}(1)$-equivariant lift of $M$ must be of the form $(x, e^{2\pi i s})\mapsto (Mx, e^{2\pi i (s + h(x))})$ for some $h\in \mathcal{C}^\infty(\mathbb{R}^{2n}, \mathbb{R})$ with the right properties. If such a lift preserves a connection of the form $\alpha_{\tilde{\kappa}'}$ for some $\tilde{\kappa}'\in \mathbb{R}^{2n}$, then a calculation analogous to computation \eqref{eq:preserve-connection-calc} shows  that $\nabla h$ is constant, so $h$ must be an affine function. Since we are ignoring global phases, we can take $h$ to be linear. Then, the fact that the lift is well defined on $\Gamma\setminus  \mathbb{H}^{\textup{red}}_{n}$ imposes $h(x)=\frac{1}{2}\sigma(\varphi_M, Mx)$, yielding the uniqueness result. Another calculation like \eqref{eq:preserve-connection-calc} also yields the last statement of the proposition.
\end{proof}





Let us put together the existence and uniqueness statements from this section. Prequantum bundles over $(\mathbb{T}^{2n}, \omega)$ exist, and they are equivalent to exactly one of $(\Gamma\setminus \mathbb{H}^{\text{red}}_{n}, \pi, \alpha_\kappa)$ with $\kappa\in [0,1)^{2n}$. For any hyperbolic $M\in \textup{Sp}(2n,\mathbb{Z})$, there is a unique $\tilde{\kappa}\in \mathbb{R}^{2n}$ and a lift $\widetilde{M}$ (unique up to global phase) such that $\alpha_{\tilde{\kappa}}$ is preserved by $\widetilde{M}$. As in the proof of Proposition \ref{proposition:classification}, an explicit gauge $g_f$ sends $\alpha_{\tilde{\kappa}}$ to $\alpha_\kappa$, where $\kappa\in [0, 1)^{2n}$ is the fractional part of $\tilde{\kappa}\in \mathbb{R}^{2n}$. The conjugated lift $g_f\widetilde{M}g_f^{-1}$ then preserves $\alpha_\kappa$.

Any other prequantum cat map on $(\Gamma\setminus \mathbb{H}^{\text{red}}_{n}, \pi)$ must preserve a principal $\text{U}(1)$-connection equivalent to one of $\alpha_\kappa$ with $\kappa\in [0,1)^{2n}$. Then, we can conjugate it via a gauge transformation to the unique (up to global phase) prequantum cat map preserving $\alpha_{\tilde{\kappa}}$ with $\tilde{\kappa}\in \mathbb{R}^{2n}$ satisfying equation \eqref{eq:prequantum-condition}. It follows in particular that $\kappa$ must be the fractional part of $\tilde{\kappa}$.

Conjugation by gauge transformations does not affect the prequantum transfer operator since the latter commutes with the action of $\text{U}(1)$, so it is enough for us to consider the prequantum cat maps explicitly described in Proposition \ref{prop:explicit-construction-map}. 

\section{The relationship between prequantum and quantum Hilbert spaces}


We turn our attention to the functions on the prequantum bundle $\Gamma \setminus \mathbb{H}^{\text{red}}_n$, which play a fundamental role in the description of the dynamics. This effort will turn out to surface an intimate link between the prequantum and quantum Hilbert spaces.

We view a function in $\mathcal{C}^\infty(\Gamma\setminus \mathbb{H}^{\text{red}}_{n})$ as a smooth $\Gamma$-invariant function on $\mathbb{H}^{\text{red}}_{n}$. Unravelling the definitions, this is a function $u\in \mathcal{C}^\infty(\mathbb{R}^{2n}\times \mathbb{S}^1)$ satisfying
\begin{equation}\label{eq:pre-condition}
u\left(x+w, e^{2\pi i\left(s+\frac{1}{2} Q(w)+\frac{1}{2} \sigma( w, x)\right)}\right)=u\left(x, e^{2\pi is}\right)\quad \text{ for all }w\in \mathbb{Z}^{2n}.
\end{equation}
Belonging to the subspace $\mathcal{C}^\infty_N(\Gamma\setminus \mathbb{H}^{\text{red}}_{n})$ imposes the additional restriction
\begin{equation*}
u\left(x, e^{2\pi is}\right)=e^{2\pi iNs}u(x, 1).
\end{equation*}

The Lebesgue measure is left and right translation invariant on $\mathbb{H}^{\text{red}}_{n}$. Hence, it is the Haar measure on $\mathbb{H}^{\text{red}}_{n}$ and the group is unimodular. It is also the unique Haar measure such that $\Gamma\setminus \mathbb{H}^{\text{red}}_{n}$ has total measure one. We hence suppress it from the notation in what follows. We endow $\mathcal{C}^\infty(\Gamma\setminus \mathbb{H}^{\text{red}}_{n})$ with the inner product obtained by integrating over a fundamental domain for $\Gamma$ such as the unit cube $[0,1)^{2n+1}$, yielding the Hilbert space $L^2(\Gamma\setminus \mathbb{H}^{\text{red}}_{n})$.


\begin{definition}
For each $N\in \mathbb{Z}$, the \emph{prequantum Hilbert space} $\widetilde{\mathcal{H}}_N$ is defined to be the $L^2$ completion of $\mathcal{C}^\infty_N(\Gamma\setminus \mathbb{H}^{\text{red}}_{n})$.
\end{definition}

We have already implicitly noted the orthogonal decomposition
\begin{equation*}
L^2(\Gamma\setminus \mathbb{H}^{\text{red}}_{n})=\bigoplus_{N\in \mathbb{Z}} \widetilde{\mathcal{H}}_{N}.
\end{equation*}

\subsection{The Schrödinger representation} 


How may one construct a function on the prequantum bundle $\Gamma\setminus \mathbb{H}^{\text{red}}_n$? We take our hint from the condition \eqref{eq:pre-condition}. Given a representation $\rho$ of $\mathbb{H}^{\text{red}}_n$ on the space of Schwartz functions $\mathcal{S}(\mathbb{R}^n)$ and a tempered distribution $\nu\in \mathcal{S}'(\mathbb{R}^n)$ that is $\Gamma$-invariant in the sense that
\begin{equation}\label{eq:gamma-invariant}
\langle\nu, f\rangle_{\mathcal{D}'(\mathbb{R}^n)}=\langle \nu, \rho\left(w, e^{\pi i Q(w)}\right) f\rangle_{\mathcal{D}'(\mathbb{R}^n)} \quad \text{ for all }w\in \mathbb{Z}^{2n}, \, f\in \mathcal{S}(\mathbb{R}^n),
\end{equation}
then we can obtain a family of functions in $\mathcal{C}^\infty(\Gamma\setminus  \mathbb{H}^{\text{red}}_n)$ by considering
\begin{equation*}
(x, e^{2\pi i s})\mapsto \langle \nu, \rho(x, e^{2\pi i s}) f\rangle_{\mathcal{D}'(\mathbb{R}^n)}
\end{equation*}
for any $f\in \mathcal{S}(\mathbb{R}^n)$. To land in $\mathcal{C}_N^\infty(\Gamma\setminus  \mathbb{H}^{\text{red}}_n)$, we require that $\rho(0,e^{2\pi i s})=e^{2\pi i N s}I$. 

 We are thus naturally led to consider the Schrödinger representation $\rho_N$ of $\mathbb{H}^{\text{red}}_n$ on $L^2(\mathbb{R}^n)$. Recall (see \cite[Section 1.3]{folland89}) that this representation, defined by
\begin{equation}\label{eq:schrodinger-representation}
\rho_N(x, e^{2\pi i s})f(y)\coloneq e^{2\pi i N \left(s+\langle p,y\rangle-\frac{1}{2}Q(x)\right)}f(y-q), \quad  f\in \mathcal{S}(\R^n),
\end{equation}
is unitary, faithful and irreducible for $N\geq 1$. We also have the characteristic commutator formula
\begin{equation}\label{eq:translation-commutator}
\rho_N(x, e^{2\pi i s})\rho_N(x', e^{2\pi i s'})=e^{2\pi i N \sigma (x, x')}\rho_N(x', e^{2\pi i s'})\rho_N(x, e^{2\pi i s}).
\end{equation}


As a brief aside, which we shall revisit at a later time, we further we note that the Schrödinger representation is intimately related to quantum translations. Indeed, after extending the Schrödinger representation to $\mathcal{S}'(\mathbb{R}^n)$ by duality,  each operator
\begin{equation*}
\rho_N(x)\coloneq \rho_N(x,1)
\end{equation*}
is called a  \emph{quantum translation}.  Although the map $x\mapsto \rho_N(x)$ is not a group homomorphism because
\begin{equation*}
\rho_N(x)\rho_N(x')=e^{\pi i N \sigma(x, x')}\rho_N(x+x'),
\end{equation*}
there are two main reasons we still speak of quantum translations. Both justifications rely on the key observation that, if we set the semiclassical parameter
\begin{equation*}\label{eq:semiclassical}
h\coloneq \dfrac{1}{2\pi N},
\end{equation*} then by \cite[Theorem 4.7]{zworski12} we have
\begin{equation*}
\rho_N(x)=\text{Op}_h\left(\exp\left(\frac{i}{h}\sigma(x,z)\right)\right)=\exp\left(-\dfrac{i}{h}\text{Op}_h(\sigma(z,x))\right).
\end{equation*}

Since classical translation by $x\in \mathbb{R}^{2n}$ is the time-$1$ flow generated by the Hamiltonian $z\mapsto\sigma(z,x)$, it is reasonable to call $\rho_N(x)$ a quantum translation as it is the time-$1$ propagator generated by the corresponding quantum Hamiltonian.

Another reason is that $\rho_N(x)$ satisfies the following exact Egorov property:
\begin{equation}\label{eq:egorov}
\rho_N(x)^{-1}\text{Op}_h(a)\rho_N(x)=\text{Op}_h(a(\cdot+x))\quad \text{ for all } a\in S(1).
\end{equation}
We refer the reader to \cite{zworski12} for any necessary refreshers on Weyl quantization on $\mathbb{R}^n$, symbol classes such as $S(1)$, and standard results in semiclassical analysis.


\subsection{Appearance of the quantum Hilbert space} Building on the previous section, we can slightly  relax the condition \eqref{eq:gamma-invariant} and still obtain functions on the prequantum bundle. Indeed, it would be enough if there existed $\theta\in \mathbb{R}^{2n}$ such that
\begin{equation*}
\rho_N\left(w, e^{\pi i Q(w)}\right)\nu=e^{2\pi i \sigma(\theta,w) }\nu\quad \text{ for all } w\in \mathbb{Z}^{2n}
\end{equation*}
because then, by the commutator formula \eqref{eq:translation-commutator}, we could again obtain functions in  $\mathcal{C}^\infty_N(\Gamma\setminus \mathbb{H}^{\text{red}}_n)$ by considering
\begin{equation}\label{eq:intertwin-orig}
(x, e^{2\pi i s})\mapsto \left\langle\rho_N\left(\tfrac{\theta}{N}\right) \nu, \rho_N(x, e^{2\pi i s}) f\right\rangle_{\mathcal{D}'(\mathbb{R}^n)}.
\end{equation}
It hence behoves us to understand the following subspace.
\begin{definition}\label{def:quantum-hilbert}
Given $N\in \mathbb{N}^*$ and $\theta\in \mathbb{T}^{2n}$, the \emph{quasi-periodic distributions} are
\begin{equation*}
\mathcal{H}_{N, \theta} \coloneq \left\{\nu\in \mathcal{S}'(\mathbb{R}^n)\mid \rho_N\left(w, e^{\pi iQ(w)}\right)\nu=e^{2\pi i \sigma(\theta, w)}\nu\quad \text{for all } w\in \mathbb{Z}^{2n} \right\}.
\end{equation*}
\end{definition}
Intuitively, these are $\mathbb{Z}^n$-periodic (up to a phase) tempered distributions on $\mathbb{R}^n$   whose $h$-Fourier transform is also $\mathbb{Z}^n$-periodic (again, up to phase).

As it turns out, once we endow $\mathcal{H}_{N,\theta}$ with an inner product, we obtain nothing short of the usual quantum Hilbert space, sometimes also called the space of quantum states. To see how we may define an inner product, we start by noticing that it is finite dimensional. In what follows, we shall write
\begin{equation*}
\mathbb{Z}_N\coloneq \{0,\dotsc, N-1\}.
\end{equation*}

\begin{lemma}
The space $\mathcal{H}_{N, \theta} $ is $N^n$-dimensional with a basis given by $\{e^{\theta}_j\}_{j\in \mathbb{Z}^n_N}$, where for $\theta=(\theta_1, \theta_2)\in \mathbb{R}^{2n}$ we define
\begin{equation}\label{eq:basis}
e_j^{\theta}(y)\coloneq N^{-n/2}\sum_{k\in \mathbb{Z}^n}e^{-2\pi i \langle \theta_2, k\rangle}\delta\left(y-k-\frac{j-\theta_1}{N}\right).
\end{equation}

\end{lemma}
\begin{remark} The tempered distributions $e^\theta_j$ satisfy the identities
\begin{equation}\label{eq:orthonormal-basis-relationships}
\begin{aligned}
e^{\theta+w}_j&=e^{\theta}_{j-m}\quad \text{for all } w=(m,n)\in \mathbb{Z}^{2n},\\
e^{\theta}_{j+Nl}&=e^{2\pi i \langle \theta_2, l\rangle}e^{\theta}_j\quad \text{for all } l\in \mathbb{Z}^n.
\end{aligned}
\end{equation}
\end{remark}

We refer to the proof of \cite[Lemma 2.5]{dyatlov24} for justification. While the space $\mathcal{H}_{N,\theta}$ is canonically defined for $\theta\in \mathbb{T}^{2n}$, the basis $\{e^{ \theta}_j\}_{j\in \mathbb{Z}^n_N}$ depends on a choice of representative $\theta_1\in \mathbb{R}^n$. We may turn each $\mathcal{H}_{N,\theta}$ into a Hilbert space if we define an inner product characterized by making $\{e^{ \theta}_j\}_{j\in \mathbb{Z}^n_N}$ an orthonormal basis. As follows from the identities \eqref{eq:orthonormal-basis-relationships}, this choice of inner product depends only on $\theta\in \mathbb{T}^{2n}$.

By the commutator formula \eqref{eq:translation-commutator} and our choice of inner product on $\mathcal{H}_{N,\theta}$, each quantum translation $\rho_N(x)$ is a unitary transformation from $\mathcal{H}_{N, \theta}$ onto $\mathcal{H}_{N, \theta-Nx}$. Since we will use this observation a few times, we introduce some notation.
\begin{definition}
For $N\in \mathbb{N}^*$ and $\theta\in \mathbb{T}^{2n}$, we denote by $T_{N, \theta}:\mathcal{H}_{N, \theta}\to \mathcal{H}_{N, 0}$ the unitary transformation given by
\begin{equation}\label{eq:hilbert-translation}
T_{N, \theta}\coloneq\rho_N\left(\dfrac{\{\theta\}}{N}\right),
\end{equation}
where $\{\theta\}\in [0,1)^{2n}$ denotes the fractional part of any representative of $\theta$ in $\mathbb{R}^{2n}$.
\end{definition}

\subsection{Decomposition of the prequantum Hilbert spaces} 

Having unveiled the nature of the $\Gamma$-invariant tempered distributions using the Schrödinger representation, we elevate the observation \eqref{eq:intertwin-orig} to a definition.
\begin{definition}
For $N\in \mathbb{N}^*$ and $\theta\in \mathbb{T}^{2n}$, we denote by
\begin{equation*}
W_{N,\theta}: \mathcal{H}_{N, \theta}\otimes \mathcal{S}(\mathbb{R}^n)\to \mathcal{C}^\infty_N( \Gamma\setminus \mathbb{H}^{\text{red}}_n)
\end{equation*}
the pairing characterized by
\begin{equation*}
\nu\otimes f\mapsto N^{n/2}\left\langle T_{N, \theta}\nu, \rho_N(\cdot)f\right\rangle_{\mathcal{D}'(\mathbb{R}^n)}.
\end{equation*}
\end{definition}


Using equations \eqref{eq:schrodinger-representation} and  \eqref{eq:basis}, we note that, for each $f\in \mathcal{S}(\mathbb{R}^n)$, we have
\begin{equation}\label{eq:unraveled-definition}
\begin{split}
W_{N,0}(e^{0}_j\otimes f)(x, e^{2\pi i s})&=\sum_{k\in \mathbb{Z}^n}\rho_N(x, e^{2\pi i s})f\left(k+\frac{j}{N}\right)\\
&=e^{2\pi i N (s+\frac{1}{N}\langle p, j\rangle-\frac{1}{2}Q(x))}\sum_{k\in \mathbb{Z}^n}e^{2\pi i N \langle p, k\rangle}f\left(k+\frac{j}{N}-q\right).
\end{split}
\end{equation}
Thus, the map $W_{N,0}$ is closely related to the Fourier--Wigner and Weil--Brazin transforms, well-known objects used in the harmonic analysis of Heisenberg nilmanifolds (see for instance \cite[Section 1.4]{folland89} or \cite[Section3]{folland04} for a discussion).

The next result captures the idea that the prequantum Hilbert spaces $\widetilde{\mathcal{H}}_N$, which are infinite-dimensional, contain the finite-dimensional quantum Hilbert spaces $\mathcal{H}_{N,\theta}$.


 
\begin{proposition}\label{proposition:unitary-identification}
Each operator $W_{N, \theta}$ initially defined on $\mathcal{H}_{N,\theta}\otimes \mathcal{S}(\mathbb{R}^n)$ extends to $\mathcal{H}_{N,\theta}\otimes L^2(\mathbb{R}^n)$ as a unitary transformation onto $\widetilde{\mathcal{H}}_N$.
\end{proposition}


\begin{proof}
It suffices to consider $\theta=0$ thanks to the following commutative diagram.
\begin{center}
\begin{tikzcd}
\mathcal{H}_{N, \theta}\otimes L^2(\mathbb{R}^n) \arrow{dr}[swap]{W_{N, \theta}} \arrow{rr}{T_{N, \theta}\otimes I} && \mathcal{H}_{N, 0}\otimes L^2(\mathbb{R}^n) \arrow{dl}{W_{N,0}}\\
& \widetilde{\mathcal{H}}_N
\end{tikzcd}
\end{center}

We first show the existence of a unique bounded extension. For $f, g\in \mathcal{S}(\mathbb{R}^n)$, equation \eqref{eq:unraveled-definition} and Fubini's theorem give us
\begin{equation*}
\begin{split}
\langle W_{N,0}(e^{0}_{j}\otimes f), W_{N,0}(e^{0}_{l}\otimes g)\rangle_{\widetilde{\mathcal{H}}_N}&=N^n\left\langle \langle e^{ 0}_{j}, \rho_N(\cdot )f\rangle_{\mathcal{D}'(\mathbb{R}^n)}, \langle e^{0}_{ l}, \rho_N(\cdot )g\rangle_{\mathcal{D}'(\mathbb{R}^n)}\right\rangle_{L^2([0,1)^{2n})}\\
&=\sum_{k, k'\in \mathbb{Z}^n}\int_{[0,1)^{n}}e^{2\pi i  \langle p, N(k-k')+j-l\rangle}\, dp\\
&\qquad\qquad\int_{[0,1)^n}f\left(k+\frac{j}{N}-q\right)   \overline{g}\left(k'+\frac{l}{N}-q\right)\, dq\\
&=\delta_{j=l}\sum_{k\in \mathbb{Z}^n}\int_{[0,1)^n-k}f\left(\frac{j}{N}-q\right)\overline{g}\left(\frac{j}{N}-q\right)\, dq\\
&=\delta_{j=l}\int_{\mathbb{R}^n}f\left(\frac{j}{N}-q\right)\overline{g}\left(\frac{j}{N}-q\right)\, dq\\
&=\left\langle e^{0}_{j}, e^{0}_{l}\right\rangle_{\mathcal{H}_{N, 0}}\langle f, g\rangle_{L^2(\mathbb{R}^n)}.
\end{split}
\end{equation*}

This also shows that $W_{N, 0}$ is unitary. To see why it is surjective onto $\widetilde{\mathcal{H}}_{N}$, we want to show that each $u\in \widetilde{\mathcal{H}}_N$ can be written as $u=\sum_{j\in \mathbb{Z}^n_N} W_{N,0}(e^{0}_j \otimes f_j)$ with $f_j\in L^2(\mathbb{R}^n)$. In what follows, we adapt the proof of \cite[Theorem 3.6]{thangavelu09}.




Let $v(x)\coloneq u(p,-q,1)$. For each $j\in \mathbb{Z}^n_N$ and $k\in \mathbb{Z}^n$, we further introduce
\begin{equation*}
v_{j, k}(x)\coloneq e^{-\pi i \langle k , p\rangle}e^{-\frac{2\pi i }{N}\langle j, k\rangle}v\left(q+\frac{k}{N},p\right)\quad \text{ and }\quad v_j\coloneq N^{-n}\sum_{k\in \mathbb{Z}^n_N}v_{j, k}.
\end{equation*}
This gives us $v=\sum_{j\in \mathbb{Z}^n_N}v_{j}$ because
\begin{equation*}
\sum_{j,k\in\mathbb{Z}^n_N } v_{j,k}(x)=\sum_{k\in \mathbb{Z}^n_N}e^{-\pi i \langle k , p\rangle}v\left(q+\frac{k}{N},p\right)\left(\sum_{j\in \mathbb{Z}^n_N}e^{- \frac{2\pi i}{N}\langle j , k\rangle}\right)=N^nv(x).
\end{equation*}
By the $\Gamma$-invariance of $u$ described in equation \eqref{eq:pre-condition}, we note that
\begin{equation}\label{eq:reduced-gamma-invariance}
v(x+w)=e^{\pi i N (Q(w)+\sigma(x,w))}v(x)\quad \text{ for all }w\in \mathbb{Z}^n.
\end{equation}
A quick calculation shows that each $v_j$ satisfies the same property. 
We claim that
\begin{equation}\label{eq:periodic-piece}
v_j\left(q+\frac{l}{N},p\right)=e^{\pi i \langle l , p\rangle}e^{\frac{2\pi i }{N}\langle j,l\rangle}v_j(x) \quad \text{ for all } l\in \mathbb{Z}^n.
\end{equation}
To see why this is true, note that applying the property \eqref{eq:reduced-gamma-invariance} yields
\begin{equation*}
\begin{split}
v_{j,k+Nl}(x)&=e^{-\pi i \langle k+Nl, p\rangle}e^{-\frac{2\pi i}{N}\langle j, k\rangle}v\left(q+\frac{k}{N}+l,p\right)\\
&=e^{-\pi i \langle k+Nl, p\rangle}e^{-\frac{2\pi i}{N}\langle j, k\rangle}e^{\pi i N \langle l, p\rangle}v\left(q+\frac{k}{N},p\right)\\
&=v_{j,k}(x),
\end{split}
\end{equation*}
and hence
\begin{equation*}
\sum_{k\in \mathbb{Z}^n_N}v_{j,k+l}=\sum_{k\in \mathbb{Z}^n_N}v_{j,k}.
\end{equation*}
This allows us to write
\begin{equation*}
\begin{split}
v_j\left(q+\frac{l}{N},p\right)&=N^{-n}\sum_{k\in \mathbb{Z}^n_N}e^{-\pi i \langle k, p\rangle}e^{-\frac{2\pi i}{N}\langle j, k\rangle}v\left(q+\frac{k+l}{N},p\right)\\
&=e^{\pi i\langle l, p\rangle}e^{\frac{2\pi i }{N}\langle j,l\rangle}N^{-n}\sum_{k\in \mathbb{Z}^n_N}v_{j,k+l}(x)\\
&=e^{\pi i \langle l , p\rangle}e^{\frac{2\pi i }{N}\langle j,l\rangle}v_j(x),
\end{split}
\end{equation*}
as desired. Property \eqref{eq:periodic-piece} implies that the function
\begin{equation*}
g_j(x)\coloneq e^{-2\pi i \langle j, q\rangle}e^{-\pi i N Q(x)}v_j(x)
\end{equation*}

is $\frac{1}{N}$-periodic in the $q$ variables. As a result, it admits a decomposition of the form
\begin{equation*}
g_j(x)=\sum_{k\in \mathbb{Z}^n}c_k(p)e^{2\pi i N\langle k, q\rangle},
\end{equation*}
where each $c_k$ is the Fourier mode given by
\begin{equation*}
c_k(p)\coloneq \int_{\left[0,\frac{1}{N}\right)^n}e^{-2\pi i N \langle k, q\rangle}g_j(x)\, dq.
\end{equation*}
Since property \eqref{eq:reduced-gamma-invariance} implies that $g_j(q, p-k)=g_j(x)e^{2\pi i N \langle k, q\rangle}$, we get the relation $c_k(p-k)=c_0(p)$. This leads to
\begin{equation*}
v_j(x)=e^{2\pi i\langle j, q\rangle}e^{\pi i N Q(x)}\sum_{k\in \mathbb{Z}^n}c_0(p+k)e^{2\pi i N \langle k , q\rangle}.
\end{equation*}
Comparing with equation \eqref{eq:unraveled-definition}, we get $v_j(x)=U_{N,0}(e^{0}_j\otimes f_j)(p,-q,1)$,
where
\begin{equation*}
f_j(y)\coloneq c_0\left(y+\frac{j}{N} \right).
\end{equation*}
We conclude by noticing that $f_j\in L^2(\mathbb{R}^n)$ and, as we wanted,
\begin{align*}
u(x, e^{2\pi i s})&=e^{2\pi i N s}v(-p, q)\\
&=e^{2\pi i N s}\sum_{j\in \mathbb{Z}^n_N}v_j(-p,q)\\
&=\sum_{j\in \mathbb{Z}^n_N}W_{N,0}(e^{0}_j\otimes f_j)(x,e^{2\pi i s}).
\qedhere
\end{align*}
\end{proof} 


The following corollary tells us that each $W_{N, \theta}$ preserves the smooth observables.

\begin{corollary}\label{corollary:observables-preserved}
Each operator $W_{N, \theta}$ restricts to an isomorphism
\begin{equation*}
 \mathcal{H}_{N, \theta}\otimes \mathcal{S}(\mathbb{R}^n)\to \mathcal{C}^\infty_N(\Gamma\setminus \mathbb{H}^{\textup{red}}_n).
\end{equation*}
\end{corollary}
\begin{proof}
Again, it suffices to check for $W_{N, 0}$. Stepping through the proof of Proposition \ref{proposition:unitary-identification}, we note  that each $f_j$ belongs to $\mathcal{S}(\mathbb{R}^n)$ if $u\in \widetilde{\mathcal{H}}_N$ is taken to be smooth.
\end{proof}


\section{Study of the prequantum transfer operator}

We have just shown that the prequantum Hilbert space $\widetilde{\mathcal{H}}_N$ is unitarily equivalent to each tensor product $\mathcal{H}_{N, \theta}\otimes L^2(\mathbb{R}^n)$. We now turn our attention to the transfer operator associated to a prequantum cat map. In principle, there is no reason this operator, once conjugated to act on $\mathcal{H}_{N, \theta}\otimes L^2(\mathbb{R}^n)$, should behave nicely with respect to the tensor product decomposition. Nevertheless, we will show that, under the appropriate quantization condition between the classical cat map $M$ and the parameters $N$ and $\theta$,  not only does the conjugated transfer operator also decompose as a tensor product,  but the component acting on $\mathcal{H}_{N, \theta}$ is precisely the usual quantum cat map $U_{N, \theta}$.

\subsection{Quantum cat maps}\label{subsection:quantum-cat-maps}

In this section, we very briefly review quantum cat maps, referring the reader to \cite[Section 2.2]{dyatlov24} for more details. As before, we fix $N\in \mathbb{N}^*$ and set $h\coloneq (2\pi N)^{-1}$ so that we may use the two semiclassical parameters interchangeably depending on what notation is more convenient or standard.

The first building block is the quantization of observables on the torus $\mathbb{T}^{2n}$. Let us consider $a\in \mathcal{C}^\infty(\mathbb{T}^{2n})$. This is nothing but a smooth $\mathbb{Z}^{2n}$-periodic function on $\mathbb{R}^{2n}$. As a result, the observable belongs to the symbol class $S(1)$, so its Weyl quantization $\text{Op}_h(a)$ is a bounded operator on $L^2(\mathbb{R}^{n})$. The key observation is that, since $a$ is $\mathbb{Z}^{2n}$-periodic, the exact Egorov property \eqref{eq:egorov} yields the commutation relation
\begin{equation*}
\text{Op}_h(a)\rho_N(w)=\rho_N(w)\text{Op}_h(a)\quad \text{ for all }w\in \mathbb{Z}^{2n}.
\end{equation*}
Thanks to this feature, it then follows that the operator $\text{Op}_h(a)$,  which we extend to $\mathcal{S}'(\mathbb{R}^n)$ by duality as usual,  preserves each quantum Hilbert space $\mathcal{H}_{N,\theta}$. We can thus define the quantizations
\begin{equation*}
\text{Op}_{N, \theta}(a)\coloneq \left.\text{Op}_h(a)\right|_{\mathcal{H}_{N,\theta}}:\mathcal{H}_{N,\theta}\to \mathcal{H}_{N,\theta},
\end{equation*}
which depend smoothly on $\theta\in \mathbb{T}^{2n}$.

We next consider the quantization of a symplectic integral matrix $M\in \text{Sp}(2n,\mathbb{Z})$. If  $M=e^X$ for some $X\in \mathfrak{sp}(2n,\mathbb{R})$, then $M$ describes the time-$1$ dynamics of a Hamiltonian flow on $\mathbb{R}^{2n}$. Indeed, the flow generated by the quadratic Hamiltonian $H(z)\coloneq \sigma( X z, z)$ is $z(t)=e^{tX}z_0$. The quantum propagator {\small$U_N\coloneq \exp\left(-\frac{i}{h}\text{Op}_h(H)\right)$} is then a unitary operator on $L^2(\mathbb{R}^n)$ satisfying
\begin{equation}\label{eq:egorov-plus}
U_N^{-1}\text{Op}_h(a)U_N=\text{Op}_h(a\circ M)\quad \text{ for all }a\in S(1).
\end{equation}
In particular, we have the following intertwining relation with quantum translations:
\begin{equation}\label{eq:intertwining-relation}
U_N^{-1}\rho_N(x)U_N=\rho_N(M^{-1}x)\quad \text{ for all }x\in \mathbb{R}^{2n}.
\end{equation}
One consequence of this property is that
\begin{equation*}
U_N\left(\mathcal{H}_{N, \theta}\right)\subseteq \mathcal{H}_{N, M\theta+\frac{N}{2}\varphi_M}.
\end{equation*}
If we choose $\theta\in \mathbb{T}^{2n}$ such that
\begin{equation}\label{eq:quant-condition}
(I-M)\theta=\frac{N}{2}\varphi_M \mod \mathbb{Z}^{2n},
\end{equation}
we hence obtain a unitary operator
\begin{equation*}
U_{N, \theta}\coloneq \left.U_N\right|_{\mathcal{H}_{N, \theta}}:\mathcal{H}_{N, \theta}\to \mathcal{H}_{N, \theta}.
\end{equation*}

When $M$ is hyperbolic, a parameter $\theta\in \mathbb{T}^{2n}$ satisfying the quantization condition \eqref{eq:quant-condition} is guaranteed to exist, and we call $U_{N, \theta}$ a \emph{quantum cat map}. It then becomes clear from property \eqref{eq:egorov-plus} that we also have an exact Egorov property on the torus:
\begin{equation*}
U_{N, \theta}^{-1}\text{Op}_{N, \theta}(a)U_{N, \theta}=\text{Op}_{N, \theta}(a\circ M)\quad \text{ for all }a\in \mathcal{C}^\infty(\mathbb{T}^{2n}).
\end{equation*}

What happens if $M\in \text{Sp}(2n, \mathbb{Z})$ is not in the image of the Lie algebra $\mathfrak{sp}(2n, \mathbb{R})$ under the exponential map? We point out that all we really needed in order to define the quantum cat maps was property \eqref{eq:egorov-plus}. 

For each $M\in \text{Sp}(2n,\mathbb{R})$, denote by $\mathcal{M}_{N, M}$ the set of all unitary transformations $M_N:L^2(\mathbb{R}^n)\to L^2(\mathbb{R}^n)$ satisfying property \eqref{eq:egorov-plus}. By \cite[Theorem 11.9]{zworski12}, these transformations exist, are unique up to a complex phase, and preserve the spaces $\mathcal{S}(\mathbb{R}^n)$ and $\mathcal{S}'(\mathbb{R}^n)$. The set
\begin{equation*}
\mathcal{M}_N\coloneq \bigcup_{M\in \text{Sp}(2n,\mathbb{R})}\mathcal{M}_{N,M}
\end{equation*}
is a subgroup of the unitary transformations of $L^2(\mathbb{R}^n)$, called the \emph{metaplectic group}. The map $M_N\mapsto M$ is a group homomorphism $\mathcal{M}_N\to \text{Sp}(2n,\mathbb{R})$. See \cite[Chapter 4]{folland89} for details.

In the general case, we thus define our quantization procedure by using the same procedure as before but by choosing a metaplectic operator $M_N$. Since we are not concerned with the complex phase (we care about spectral properties), the specific choice does not matter. By uniqueness up to a phase, we know that the definition of the quantum cat maps using Weyl quantization agrees with the more general one.




\subsection{Factorization of the prequantum transfer operator}

We tackle the study of the prequantum transfer operator $F_N: \widetilde{\mathcal{H}}_N\to \widetilde{\mathcal{H}}_N$ associated to a classical cat map $M\in \text{Sp}(2n, \mathbb{Z})$. More precisely, seeking to better understand the dynamics generated by $F_N$, we examine the unitarily equivalent operator $W_{N,\theta}^{-1}F_NW_{N,\theta}$.
\begin{equation}\label{eq:comm-diagram0}
\begin{tikzcd}
\mathcal{H}_{N, \theta}\otimes L^2(\mathbb{R}^n)\arrow[r, dashed]\arrow{d}[swap]{W_{N, \theta}} &  \mathcal{H}_{N, \theta}\otimes L^2(\mathbb{R}^n)\\
\widetilde{\mathcal{H}}_{N}\arrow{r}[swap]{F_N} &  \widetilde{\mathcal{H}}_{N}\arrow{u}[swap]{W_{N,\theta}^{-1}}
\end{tikzcd}
\end{equation}
The explicit formulas quickly get involved and are not very elucidating.  We instead take a small turnabout. We introduce the operator $P_{N,\theta}:\mathcal{S}(\mathbb{R}^n)\to \mathcal{H}_{N, \theta}$ given by
\begin{equation}\label{eq:projection-definition}
P_{N,\theta}\coloneq\sum_{w\in \mathbb{Z}^{2n}}e^{-2\pi i\sigma(\theta, w)}\rho_N(w, e^{\pi i Q(w)}).
\end{equation}
It is well defined because the series converges in $\mathcal{S}'(\mathbb{R}^n)$ and
\begin{equation*}
\rho_N(w, e^{\pi i Q(w)})P_{N,\theta}=e^{2\pi i \sigma (\theta, w)} P_{N,\theta}\quad \text{ for all }w\in \mathbb{Z}^{2n},
\end{equation*}
so we indeed land in the quantum Hilbert space $\mathcal{H}_{N,\theta}$.

Through an explicit computation using equations \eqref{eq:schrodinger-representation} and \eqref{eq:basis} as well as the Poisson summation formula, one can check that the operator $P_{N, \theta}$ defined by equation \eqref{eq:projection-definition} can also be written as
\begin{equation}\label{eq:projection-definition-alternative}
P_{N, \theta} f = \sum_{j\in \mathbb{Z}^n_N}\langle e^{(\theta_1, -\theta_2)}_j, f\rangle_{\mathcal{D}'(\mathbb{R}^n)}e^{\theta}_j\quad \text{ for all } f\in \mathcal{S}(\mathbb{R}^n).
\end{equation}


\begin{lemma}\label{lemma-projection-surjectivity}
The operator $P_{N,\theta}:\mathcal{S}(\mathbb{R}^n)\to \mathcal{H}_{N, \theta}$ is surjective.
\end{lemma}
\begin{proof} 

The commutator formula \eqref{eq:translation-commutator} combined with equations \eqref{eq:hilbert-translation} and \eqref{eq:projection-definition}  yield
\begin{equation*}
T_{N, \theta}^{-1}P_{N, 0}T_{N, \theta}=P_{N, \theta}.
\end{equation*}
It thus suffices to consider the case $\theta=0$. We argue similarly to \cite[Lemma 2.6]{dyatlov24}. Let $\nu \in \mathcal{H}_{N, 0}$ and define
\begin{equation*}
R\nu(y)\coloneq N^{n/2}\int_{\mathbb{T}^n}\left\langle T_{N, (Ny, -\theta_2)}\nu, e^{(-Ny, \theta_2)}_0\right\rangle_{\mathcal{H}_{N, (-Ny, \theta_2)}}\, d\theta_2, \quad  y\in \mathbb{R}^n.
\end{equation*}
One can check that $R\nu \in \mathcal{S}(\mathbb{R}^n)$ using a non-stationary phase argument and the following consequence of the identities \eqref{eq:orthonormal-basis-relationships}:
\begin{equation*}
R\nu (y-k)=R(e^{2\pi i \langle \theta_2, k\rangle}\nu)(y)\quad \text{ for all } k\in \mathbb{Z}^n.
\end{equation*}
We then compute
\begin{equation*}
\begin{split}
\langle e^{0}_j, R\nu \rangle_{\mathcal{D}'(\mathbb{R}^n)}&=N^{-n/2}\sum_{k\in \mathbb{Z}^n} R\nu\left(k+\frac{j}{N}\right)\\
&=\sum_{k\in \mathbb{Z}^n}\int_{\mathbb{T}^n}\langle T_{N, (0, -\theta_2)}\nu, e^{(-Nk-j, \theta_2)}_0\rangle_{\mathcal{H}_{N, (0,\theta_2)}}\, d\theta_2\\
&=\sum_{k\in \mathbb{Z}^n}\int_{\mathbb{T}^n}e^{-2\pi i \langle \theta_2, k\rangle}\langle T_{N, (0, -\theta_2)}\nu, e^{(0, \theta_2)}_j\rangle_{\mathcal{H}_{N, (0,\theta_2)}}\, d\theta_2\\
&=\langle \nu, e^{0}_j\rangle_{\mathcal{H}_{N,0}},
\end{split}
\end{equation*}
where we have used the identities \eqref{eq:orthonormal-basis-relationships} again as well as the convergence of the Fourier series of the function  $\theta_2 \mapsto  \langle T_{N, (0, -\theta_2)}\nu, e^{(0, \theta_2)}_j\rangle_{\mathcal{H}_{(0, \theta_2)}}$.
By equation \eqref{eq:projection-definition-alternative}, this yields $P_{N,0}R\nu =\nu$, as desired.
\end{proof}

The following property will also be useful.
\begin{lemma}\label{lemma:projection-commuting}

Let $M\in \textup{Sp}(2n, \mathbb{Z})$ be hyperbolic, fix $N\in \mathbb{N}^*$ and choose $\theta\in \mathbb{T}^{2n}$ satisfying equation \eqref{eq:quant-condition}. Then, the following diagram commutes.
\begin{center}
\begin{tikzcd}
\mathcal{S}(\mathbb{R}^n)\arrow{r}{M_N}\arrow{d}[swap]{P_{N, \theta}} & \mathcal{S}(\mathbb{R}^n)\arrow{d}{P_{N, \theta}}\\
\mathcal{H}_{N,\theta}\arrow{r}[swap]{U_{N, \theta}} &  \mathcal{H}_{N,\theta}
\end{tikzcd}
\end{center}

\end{lemma}
\begin{proof}

As noted before, there always exists $\theta\in \mathbb{T}^{2n}$ satisfying condition \eqref{eq:quant-condition} because the matrix $M$ is hyperbolic by assumption. Using properties \eqref{eq:dyat} and \eqref{eq:quant-condition}, we obtain
\begin{equation*}
\begin{split}
P_{N, \theta}&=\sum_{w\in \mathbb{Z}^{2n}} e^{-2\pi i \sigma(M\theta, Mw)}\rho_N(w, e^{\pi i (Q(Mw)+\sigma(\varphi_M, Mw))})\\
&=\sum_{w\in \mathbb{Z}^{2n}} e^{-2\pi i (\sigma(\theta, Mw)-\frac{1}{2}\sigma(\varphi_M, Mw))}\rho_N(w, e^{\pi i (Q(Mw)+\sigma(\varphi_M, Mw))})\\
&=\sum_{w\in \mathbb{Z}^{2n}} e^{-2\pi i \sigma(\theta, Mw)}\rho_N(w, e^{\pi i Q(Mw)}).
\end{split}
\end{equation*}
Therefore, using the intertwining relation \eqref{eq:intertwining-relation} yields
\begin{equation*}
\begin{split}
U_{N,\theta} P_{N, \theta}&=\sum_{w\in \mathbb{Z}^{2n}} e^{-2\pi i \sigma(\theta, Mw)} M_N\rho_N(w, e^{\pi i Q(Mw)})\\
&=\sum_{w\in \mathbb{Z}^{2n}} e^{-2\pi i \sigma(\theta, Mw)} \rho_N(Mw, e^{\pi i Q(Mw)})M_N\\
&=P_{N,\theta}M_N.\qedhere
\end{split}
\end{equation*}
\end{proof}

Armed with this, we then want to find an operator on $\mathcal{S}(\mathbb{R}^n)\otimes L^2(\mathbb{R}^n)$ making the following diagram commute.
\begin{equation}\label{eq:comm-diagram}
\begin{tikzcd}
\mathcal{S}(\mathbb{R}^n)\otimes L^2(\mathbb{R}^n)\arrow{d}[swap]{P_{N,\theta}\otimes I}\arrow[dashed]{r} & \mathcal{S}(\mathbb{R}^n)\otimes L^2(\mathbb{R}^n)\arrow{d}{P_{N,\theta}\otimes I}\\
\mathcal{H}_{N, \theta}\otimes L^2(\mathbb{R}^n)\arrow{d}[swap]{W_{N,\theta}} &  \mathcal{H}_{N, \theta}\otimes L^2(\mathbb{R}^n)\arrow{d}{W_{N,\theta}}\\
\widetilde{\mathcal{H}}_{N}\arrow{r}[swap]{F_N} &  \widetilde{\mathcal{H}}_{N}
\end{tikzcd}
\end{equation}
We do this when $\varphi_M=0$. The trick lies in using metaplectic operators.

\begin{proposition}\label{proposition:bar-decomp}

If $\varphi_M=0$, then the operator $M_N\otimes M_{N}$ makes the diagram \eqref{eq:comm-diagram} with $\theta = 0$ commute.




\end{proposition}

\begin{proof}
Since $U_{N, 0}$ is a unitary transformation of the finite-dimensional space $\mathcal{H}_{N, 0}$, we can use an orthonormal basis of eigenvectors to check that
\begin{equation}\label{eq:complex-commute}
KU_{N, 0}K=U_{N, 0},
\end{equation}
where $K$ is the complex conjugation operator. Indeed, the complex conjugate of an eigenvector to a given eigenvalue is an eigenvector for the conjugate eigenvalue.


We recall that, from Proposition \ref{prop:explicit-construction-map}, we have
\begin{equation*}
(F_N u)(x, e^{2\pi i s})=u\left(M^{-1}x, e^{2\pi i s}\right).
\end{equation*}
Therefore, for all $f,g\in \mathcal{S}(\mathbb{R}^n)$, we can apply the intertwining relation \eqref{eq:intertwining-relation}, the fact that $M_N$ is unitary on $L^2(\mathbb{R}^n)$, property \eqref{eq:complex-commute} and Lemma \ref{lemma:projection-commuting} to write
\begin{equation*}
\begin{split}
F_{N}(W_{N,0}(P_{N,0}f\otimes g))(x,e^{2\pi i s})&=N^{n/2}\left\langle P_{N, 0}f,\rho_N\left(M^{-1} x,  e^{2\pi i  s}\right)g\right\rangle_{\mathcal{D}'(\mathbb{R}^{n})} \\
&=N^{n/2}\left\langle P_{N, 0}f,M_{N}^{-1}\rho_N\left( x, e^{2\pi i s}\right)M_N g\right\rangle_{\mathcal{D}'(\mathbb{R}^{n})}\\
&=N^{n/2}\left\langle KM_{N}K P_{N, 0}f,\rho_N\left( x, e^{2\pi i s}\right)M_N g\right\rangle_{\mathcal{D}'(\mathbb{R}^n)}\\
&=N^{n/2}\left\langle P_{N, 0}M_N f,\rho_N\left( x, e^{2\pi i s}\right)M_N g\right\rangle_{\mathcal{D}'(\mathbb{R}^n)}.
\end{split}
\end{equation*}
This shows that
\begin{align*}
F_{N}(W_{N,0}(P_{N,0}f\otimes g))
&=W_{N,0}(P_{N,0}M_Nf\otimes M_N g).\qedhere
\end{align*}
\end{proof}



Applying Lemmas \ref{lemma-projection-surjectivity} and \ref{lemma:projection-commuting}, we thus obtain the following corollary.

\begin{corollary}\label{cor:almost-there}
If $\varphi_M=0$, then
\begin{equation*}
W_{N,0}^{-1}F_NW_{N, 0}=U_{N, 0}\otimes M_{N}
\end{equation*}
on $\mathcal{H}_{N, 0}\otimes L^2(\mathbb{R}^n)$. This completes the commutative diagram \eqref{eq:comm-diagram0} with $\theta = 0$.
\end{corollary}

\subsection{Pollicott--Ruelle spectrum of the prequantum transfer operator}

In this section, we prove Theorems \ref{theorem:main-result} and \ref{theorem:second-result}. Using the tensor product decomposition from Corollary \ref{cor:almost-there}, we show that when $\varphi_M=0$, the Pollicott--Ruelle resonances of $F_{N}$ are given by $e^{i\varphi_j} e^{i\lambda}$, where $e^{i\varphi_j}$, $\varphi_j\in \mathbb{R}$, are the eigenvalues of $U_{N, 0}$ on $\mathcal{H}_{N, 0}$ and $e^{i\lambda}$, $\lambda\in \mathbb{C}$, are the Pollicott--Ruelle resonances of the operator $M_{N}$.

The resonances of $M_{N}$ are defined via correlation functions $C_{u,v}$ as in Section \ref{section:statement-of-the-results} but with observables $u$ and $v$ in $ \mathcal{S}(\mathbb{R}^n)$. We have not yet shown that these are well defined. The strategy is to reduce the problem to a study of the transfer operator
\begin{equation*}
L_E u \coloneq u \circ E^{-1}
\end{equation*}
 associated to an expanding linear map $E:\mathbb{R}^n\to \mathbb{R}^n$. As a reminder, we say that an invertible matrix $E\in \text{GL}(n, \mathbb{R})$ is \emph{expanding} if  $\Vert E^{-1}\Vert<1$.
 
 Then, \cite[Proposition 3.4.6]{faure15} shows that $L_E$ has a well-defined discrete Pollicott--Ruelle spectrum contained in annuli \eqref{eq:band-structure-2} indexed by $k\in \mathbb{N}$. Each annulus corresponds to the restriction of $L_E$ to homogeneous polynomials on $\mathbb{R}^n$ of degree $k$.


\begin{proposition}\label{prop:essentially-there}
Let $M\in\textup{Sp}(2n,\mathbb{R})$ be hyperbolic. There exists an expanding matrix $E\in \textup{GL}(n, \mathbb{R})$ with $|\det E|>1$ such that each metaplectic operator $M_N$ on $ L^2(\mathbb{R}^n)$ is unitarily equivalent to the unitary operator $|\det E|^{-1/2}L_E$ on $L^2(\mathbb{R}^n)$. The equivalence preserves $\mathcal{S}(\mathbb{R}^n)$.
\end{proposition}

\begin{proof}
The first step of our argument is to show that there exist $D\in \textup{Sp}(2n, \mathbb{R})$ and an expanding $E\in \textup{GL}(n, \mathbb{R})$ such that
\begin{equation}\label{eq:normal-form}
D^{-1}MD=\left(\begin{matrix}
E & 0\\
0 & (E^T)^{-1}
\end{matrix}\right).
\end{equation}
To see why this is true, consider the stable and unstable subspaces $E_s, E_u\subset\mathbb{R}^{2n}$ associated to the hyperbolic matrix $M$. For any $x,y\in E_s$ (respectively $E_u$) and $n\in \mathbb{Z}$, we have $\omega(x, y)=\omega(M^nx,M^ny)$. Taking the limit $n\to \infty$ (respectively $n\to -\infty$) gives $\omega(x,y)=0$. This implies that $E_s$ and $E_u$ are isotropic. Since $\mathbb{R}^{2n}=E_s\oplus E_u$, it follows by a count of dimensions that both subspaces are Lagrangian.

Since $\text{Sp}(2n, \mathbb{R})$ acts transitively on the set of pairs of transverse Lagrangians, there exists $D\in \text{Sp}(2n, \mathbb{R})$ simultaneously mapping $E_s$ to $\mathbb{R}^{n}\oplus \{0\}$ and $E_u$ to $\{0\}\oplus \mathbb{R}^n$. The stable and unstable subspaces are preserved by $M$, which acts as a contraction on $E_s$ and an expanding map on $E_u$. This implies that we must have the  relation \eqref{eq:normal-form} for some expanding $E\in \text{GL}(n, \mathbb{R})$. It is then clear that $|\det E|>1$.

Since the map $M_N\mapsto M$ from the metaplectic group $\mathcal{M}_N$ to $\text{Sp}(2n, \mathbb{R})$ is a group homomorphism, we have $(D^{-1}MD)_N=D_N^{-1}M_ND_N$ up to a complex phase. This tells us in particular that the operator $M_N: L^2(\mathbb{R}^n)\to L^2(\mathbb{R}^n)$ is unitarily equivalent to $(D^{-1}MD)_N:L^2(\mathbb{R}^n)\to L^2(\mathbb{R}^n)$. Note that $D_N$ preserves the subspace $\mathcal{S}(\mathbb{R}^n)$.

We claim that, up to a complex phase, we have
\begin{equation*}
(D^{-1}MD)_N=|\det E|^{-1/2}L_E.
\end{equation*}
By \cite[Theorem 11.9]{zworski12}, this is equivalent to asking that
\begin{equation*}
L_E^{-1}\text{Op}_h(a)L_E =\text{Op}_h(a\circ (D^{-1}MD))\quad \text{ for all } a\in S(1).
\end{equation*}
This can be readily checked by using the definition
\begin{equation*}
[\text{Op}_h(a)u](x)\coloneq (2\pi h)^{-n}\iint e^{\frac{i}{h}\langle \xi, x-y\rangle}a\left(\tfrac{x+y}{2}, \xi\right)u(y)\, d\xi\, dy, \quad u\in \mathcal{S}(\mathbb{R}^n),
\end{equation*}
 the change of variables formula, and an argument by density.
\end{proof}

We can now collect the results of the previous sections.



\begin{proof}[Proof of Theorem \ref{theorem:main-result}]
Suppose $\varphi_M=0$ and fix $N\in \mathbb{N}^*$.
Let $u, v\in \mathcal{C}^\infty_N(\Gamma\setminus \mathbb{H}^{\text{red}}_n)$ and set $\tilde{u}\coloneq W_{N,0}^{-1}u$, $\tilde{v}\coloneq W_{N,0}^{-1}v$. By Corollary \ref{cor:almost-there}, since $W_{N, 0}$ is unitary, we get
\begin{equation}\label{eq:correlation-decomposition}
\begin{split}
C_{u, v}(t)&=\langle F_N^t u, v\rangle_{\widetilde{\mathcal{H}}_N}\\
&=\langle W_{N, 0}(U_{N, 0}\otimes M_{N})^tW_{N,0}^{-1}u, v\rangle_{\widetilde{\mathcal{H}}_N}\\
&=\langle (U_{N, 0}\otimes M_{N})^t\tilde{u}, \tilde{v}\rangle_{\mathcal{H}_{N, 0}\otimes L^2(\mathbb{R}^n)}.
\end{split}
\end{equation}
It follows that
\begin{equation*}
\widehat{C}_{u, v}(\lambda)\coloneq\sum_{t=1}^\infty e^{-i\lambda t}C_{u,v}(t)=\langle (e^{i \lambda}I-U_{N,0}\otimes M_N)^{-1} \tilde{u}, \tilde{v}\rangle_{\mathcal{H}_{N, 0}\otimes L^2(\mathbb{R}^n)}.
\end{equation*}
In light of Corollary \ref{corollary:observables-preserved},  $F_N$ has the same Pollicott--Ruelle resonances as $U_{N,0}\otimes M_N$. We claim that these resonances are the products of the eigenvalues of $U_{N,0}$ and the resonances of $M_N$. By Proposition \ref{prop:essentially-there}, the resonances of $M_N$ are the same as those of $|\det E|^{-1/2}L_E$, so this will give us the desired result.

Let $\{\nu_j \}_{j \in \mathbb{Z}^n_N}$ be an orthonormal eigenbasis of the unitary operator $U_{N, 0}$ on $\mathcal{H}_{N, 0}$. Let us write $U_{N, 0}\nu_j = e^{i\varphi_j}\nu_j$ with $\varphi_j\in [0, 2\pi )$. We have $\tilde{u} = \sum_{j\in \mathbb{Z}^n_N} a_j\nu_j\otimes f_j$ and $\tilde{v}= \sum_{j\in \mathbb{Z}^n_N} b_j \nu_j\otimes g_j$ for some $a_j, b_j\in \mathbb{C}$ and $f_j, g_j\in \mathcal{S}(\mathbb{R}^n)$. Equation \eqref{eq:correlation-decomposition} becomes
\begin{equation*}
C_{u, v}(t)=\sum_{j\in \mathbb{Z}^n_N}a_{j}\overline{b}_j e^{i\varphi_jt}\langle M_{N}^t f_j, g_j\rangle_{L^2(\mathbb{R}^n)},
\end{equation*}
so we obtain
\begin{equation*}
\widehat{C}_{u, v}(\lambda)=\sum_{j\in \mathbb{Z}^n_N}a_{j}\overline{b}_j \langle (e^{-i\varphi_j }e^{i \lambda}I-M_{N})^{-1} f_j, g_j\rangle_{L^2(\mathbb{R}^n)}.
\end{equation*}
This shows that, if $\lambda\in \mathbb{C}$ is a pole of the meromorphic extension of $\widehat{C}_{u,v}$ to $\mathbb{C}$, then $e^{i\lambda}=e^{i\varphi_j}e^{i\lambda'}$ for some eigenvalue $e^{i\varphi_j}$ of $U_{N,0}$ and some resonance $e^{i\lambda'}$ of $M_{N}$.

For the converse direction, we use the characterization of resonances as eigenvalues on anisotropic Sobolev spaces. Indeed, if $e^{i\lambda'}$ is a resonance of $M_N$, then there is an anisotropic Sobolev space on which $e^{i\lambda'}$ is an eigenvalue of $M_N$, but then taking the tensor product of this space with any eigenspace of $\mathcal{H}_{N,0}$ yields an anisotropic Sobolev space for $U_{N, 0}\otimes M_N$, with eigenvalues of the desired form. \end{proof}

\begin{proof}[Proof of Theorem \ref{theorem:second-result}]
By Proposition \ref{prop:essentially-there}, $M_N$ and $|\det E|^{-1/2}L_E$ have the same resonances. Pick $\rho\in (\Vert E^{-1}\Vert, 1)$. Then, the resonances of $|\det E|^{1/2}M_N$ lie inside the disk of radius $\rho$, except for the resonance $1$. Denote by $\Pi_{\rho}$ the spectral projector of $|\det E|^{1/2}M_N$ on $\{z\in \mathbb{C} \mid |z|>\rho\}$. It has finite rank, commutes with $M_N$, and is in fact the delta distribution $\delta_0$. Therefore, for any $f, g\in \mathcal{S}(\mathbb{R}^n)$,  we have
\begin{equation}\label{eq:important-part}
\langle (M_N\Pi_\rho)^tf, g\rangle_{L^2(\mathbb{R}^n)}=|\det E|^{-t/2} f(0)\int_{\mathbb{R}^n} \overline{g(x)}\, dx.
\end{equation}

As explained in \cite[Chapter 3]{faure15}, there exists a family of Hilbert spaces $H^r(\mathbb{R}^n)$ for arbitrarily large $r>0$, called anisotropic Sobolev spaces, satisfying
$$\mathcal{S}(\mathbb{R}^n)\subset H^r(\mathbb{R}^n)\subset \mathcal{S}'(\mathbb{R}^n),$$
such that the operator $M_N$ extends to a bounded operator
$$M_N: H^r(\mathbb{R}^n)\to H^r(\mathbb{R}^n)$$
with an essential spectral radius bounded by $\rho^r|\det E|^{-1/2} $. From the definition of $\Pi_\rho$, the essential spectral radius of $M_N(I-\Pi_\rho)$ is bounded by $\rho|\det E|^{-1/2}$, so
$$\Vert (M_N(I-\Pi_\rho))^t\Vert_{H^r(\mathbb{R}^n)}\leq C |\det E|^{-t/2}\rho^t$$
for some constant $C>0$. It follows that
\begin{equation*}
\left|\langle (M_N(I-\Pi_\rho))^tf, g\rangle_{L^2(\mathbb{R}^n)}\right|\leq C \Vert f \Vert_{H^r(\mathbb{R}^n)} \Vert g \Vert_{(H^r(\mathbb{R}^n))'}|\det E|^{-t/2}\rho^t.
\end{equation*}
Combining this with equation \eqref{eq:important-part}, we deduce the estimate
\begin{equation*}
\begin{split}
\langle M^t_{N} f, g \rangle_{L^2(\mathbb{R}^n)}&=\langle (M_N\Pi_\rho)^tf, g\rangle_{L^2(\mathbb{R}^n)}+ \langle (M_N(I-\Pi_\rho))^tf, g\rangle_{L^2(\mathbb{R}^n)}\\
&= |\det E|^{-t/2}\left(f(0)\int_{\mathbb{R}^n} \overline{g(x)}\, dx + \mathcal{O}(\rho^t)\Vert f \Vert_{H^r(\mathbb{R}^n)}\Vert g\Vert_{(H^r(\mathbb{R}^n))'}\right).
\end{split}
\end{equation*}
For any $u, v\in \mathcal{C}^\infty_N(\Gamma\setminus \mathbb{H}^{\text{red}}_n)$, let us write $\hat{u}\coloneq(I\otimes \delta_0)W^{-1}_{N,0}u$ and $\hat{v}\coloneq(I\otimes\, dx)W^{-1}_{N,0}v$, where $I\otimes dx$ is defined in equation \eqref{eq:projectors}. Then, following equation \eqref{eq:correlation-decomposition}, we obtain
\begin{equation*}
\begin{split}
C_{u, v}(t)&=\langle (U_{N, 0}\otimes M_{N})^tW_{N, 0}^{-1}u, W_{N, 0}^{-1}v\rangle_{\mathcal{H}_{N, 0}\otimes L^2(\mathbb{R}^n)}\\
&=\langle U_{N, 0}^t\hat{u}, \hat{v}\rangle_{\mathcal{H}_{N, 0}} |\det E|^{-t/2}  \left(1 + \mathcal{O}_{u,v}(\rho^t)\right).\qedhere
\end{split}
\end{equation*}
\end{proof}

\subsection{Example} Let us put together by working through an explicit example. For any $A\in \text{GL}(n, \mathbb{Z})$ with $|\det A| = 1$, the matrix
$$M= \left(\begin{matrix}
A & 0\\
0 & (A^T)^{-1}\\
\end{matrix}\right)$$
belongs to $\text{Sp}(2n, \mathbb{Z})$. By Remark \ref{remark:varphi-m}, we have $\varphi_M=0$. Fix $n=3$ and consider
$$A= \left(\begin{matrix}
0 & 1 & 0\\
0 & 0 & 1\\
1 & 1 & 0\\
\end{matrix}\right).$$
This matrix has one real eigenvalue $\mu_+>1$ and a complex conjugate pair $\mu, \overline{\mu}$ lying inside the unit disk. As illustrated in Figure \ref{figure:eigenvalues-m}, the matrix $M$ is  hyperbolic.

\begin{figure}[htp]
\includegraphics[scale=0.4]{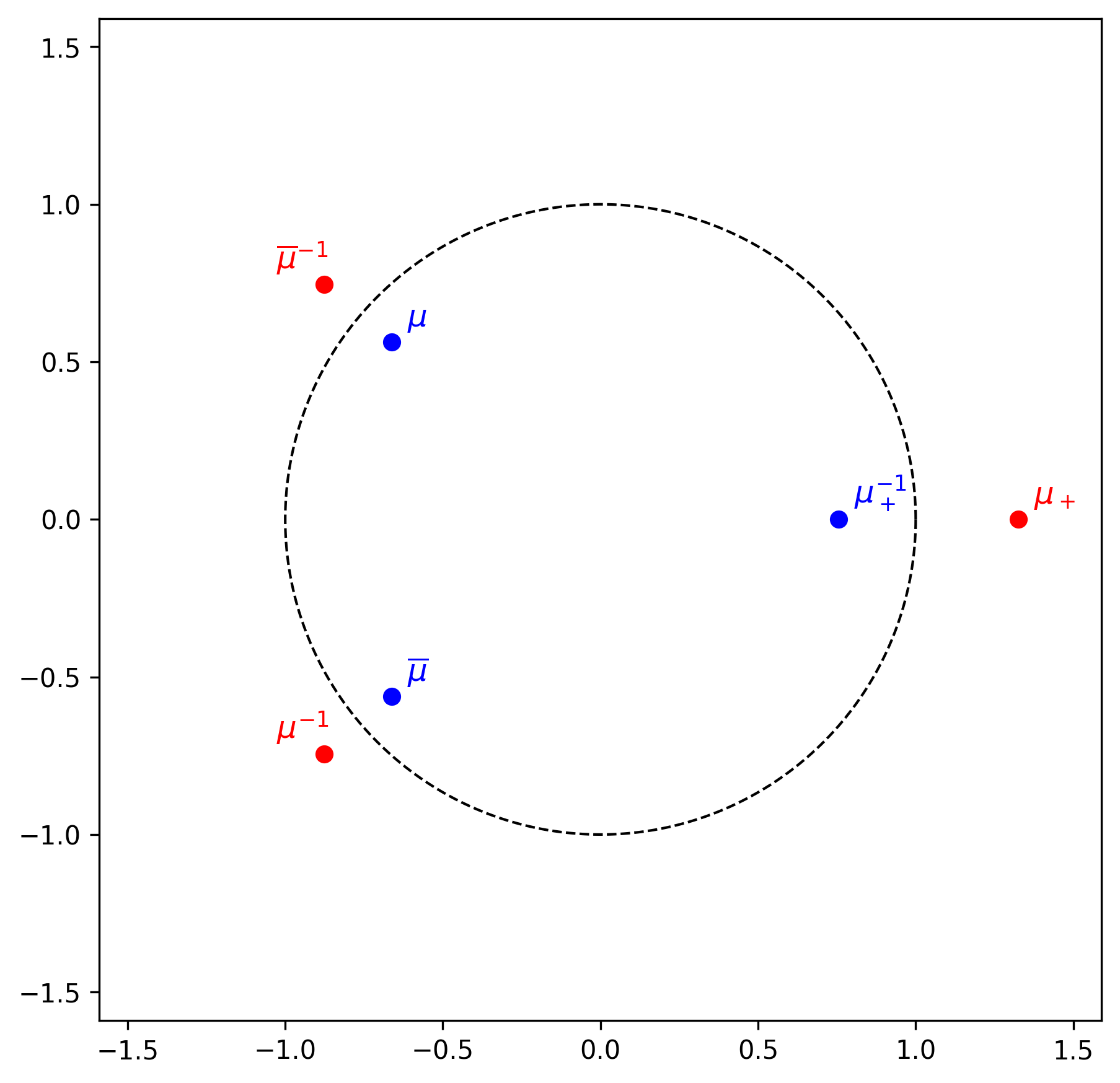}
\caption{Eigenvalues of the hyperbolic matrix $M\in \text{Sp}(6, \mathbb{Z})$.}
\label{figure:eigenvalues-m}
\end{figure}

Note that the eigenvalues $\mu_+$, $\mu^{-1},\overline{\mu}^{-1}$ lay outside of the unit disk. Using their eigenvectors, it is then easy to check that $M$ is symplectically conjugate to
$$\left(\begin{matrix}
E & 0\\
0 & (E^T)^{-1}\\
\end{matrix}\right),$$
where
$$E= \left(\begin{matrix}
\mu_+ & 0 & 0\\
0 & \text{Re}(\mu^{-1}) & \text{Im}(\mu^{-1})\\
0 & -\text{Im}(\mu^{-1}) & \text{Re}(\mu^{-1})\\
\end{matrix}\right).$$
We can then compute $\det E = \mu_+^2>1$, as well as the operator norms
$$\Vert E^{-1} \Vert  = |\mu|< 1 <  \mu_+= \Vert E\Vert.$$
Therefore, $E$ is an expanding map as in the statement of Theorem \ref{theorem:main-result}. The eigenvalues of the transfer operator $L_E$ are
$$\mu_+^{-\alpha_1}\mu^{\alpha_2}\overline{\mu}^{\alpha_3}$$
for all $\alpha_1, \alpha_2, \alpha_3\in \mathbb{N}$ satisfying $\alpha_1+\alpha_2+\alpha_3=k$.

\begin{figure}[htp]
\includegraphics[scale=0.5]{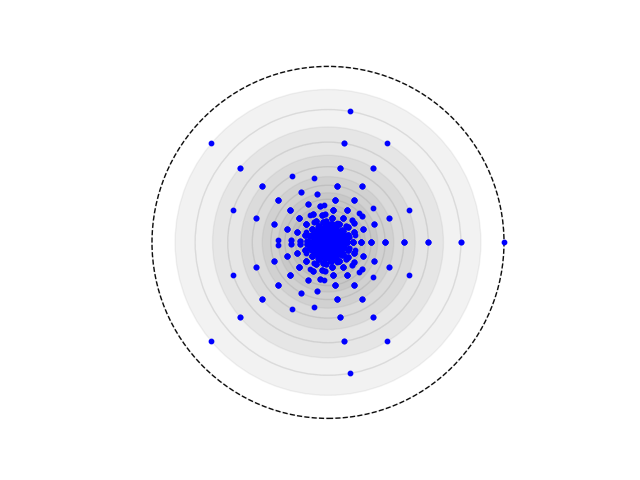}
\caption{Resonances of the operator $L_E$ and the annuli \eqref{eq:band-structure-2}.}
\end{figure}

Using \cite[Proposition 2.11]{dyatlov24}, we see that for any $N\in \mathbb{N}^*$ and $\theta=0$, a quantization $U_{N, 0}$  of $M$ acting on $\mathcal{H}_{N,0}$ satisfies
$$U_{N, 0} e^0_j =e^0_{Aj} \quad \text{ for all } j\in \mathbb{Z}^n_N.$$
Pick $N=2$, for instance. By choosing the right ordering of the basis $\{e^0_j\}_{j\in \mathbb{Z}^3_2}$, the matrix representation of $U_{2,0}:  \mathcal{H}_{2, 0}\to \mathcal{H}_{2, 0}$ is
$$M_{2,0}=\left(\begin{matrix}
1 & 0 & 0 & 0 & 0 & 0 & 0 & 0\\
0 & 0 & 0 & 0 & 0 & 0 & 0 & 1\\
0 & 1 & 0 & 0 & 0 & 0 & 0 & 0\\
0 & 0 & 1 & 0 & 0 & 0 & 0 & 0\\
0 & 0 & 0 & 1 & 0 & 0 & 0 & 0\\
0 & 0 & 0 & 0 & 1 & 0 & 0 & 0\\
0 & 0 & 0 & 0 & 0 & 1 & 0 & 0\\
0 & 0 & 0 & 0 & 0 & 0 & 1 & 0\\
\end{matrix}\right).$$
This is nothing but a permutation matrix with two cycles. Its spectrum is
$$\sigma(U_{2,0})=\{e^{2\pi i k/7} \mid k=0,1, \dotsc, 6\}.$$
See Figure \ref{figure:spectra-b}.
The eigenvalue $1$ has double multiplicity because it has eigenvectors
$$e_0^0\qquad \text{ and }\qquad \sum_{j\in \mathbb{Z}^3_2}e_j^0.$$
By Theorem \ref{theorem:main-result}, we obtain the prequantum resonances illustrated in Figure \ref{figure:spectra-a}.

\bibliographystyle{alpha}
\bibliography{references}

@article{faure07,
	title = {Prequantum chaos: Resonances of the prequantum cat map},
	journal = {Journal of Modern Dynamics},
	volume = {1},
	number = {2},
	pages = {255--285},
	year = {2007},
	author = {Faure, Frédéric},
}

@article{faure15,
	title={{Prequantum transfer operator for symplectic Anosov diffeomorphism}},
	author = {Faure, Frédéric and Tsujii, Masato},
	journal = {{Astérisque}},
	number = {375},
	volume={2015},
	pages={1--237},
    year = {2015}
}

@article{faure24,
	title={{Micro-local analysis of contact Anosov flows and band structure of the Ruelle spectrum}},
	author = {Faure, Frédéric and Tsujii, Masato},
	year={2024},
	journal={Communications of the American Mathematical Society},
	pages={641--745},
	volume={4},
	number={1}
}

@article{dyatlov24,
	title={{Semiclassical measures for higher dimensional quantum cat maps}},
	author={Dyatlov, Semyon and Jézéquel, Malo},
	journal={Annales Henri Poincaré},
	volume={25},
	number={2},
	pages={1545--1605},
	year={2024}
}

@article{zelditch97,
	title={{Index and dynamics of quantized contact transformations}},
	author={Zelditch, Steven},
	journal={Annales de l'Institut Fourier},
	volume={47},
	year={1997},
	number={1},
	pages={305--363}
}

@book{folland89,
	title={{Harmonic Analysis in Phase Space}},
	author={Folland, Gerald B.},
	publisher = {Princeton University Press},
	year = {1989},
	volume={122},
	address={Princeton, NJ},
	series={Annals of Mathematics Studies}
}

@article{folland04,
	title={{Compact Heisenberg manifolds as CR manifolds}},
	author={Folland, Gerald B.},
	journal={Journal of Geometric Analysis},
	volume={14},
	number={3},
	year={2004},
	pages={521--532}
}

@article{thangavelu09,
	author = {Thangavelu, Sundaram},
	year = {2009},
	pages = {75--93},
	title = {{Harmonic analysis on Heisenberg nilmanifolds}},
	volume = {50},
	number={2},
	journal = {Revista de la Unión Matemática Argentina}
}

@book{zworski12,
	author={Zworski, Maciej},
	title={{Semiclassical Analysis}},
	publisher={American Mathematical Society},
	year={2012},
	volume={138},
	series={Graduate Studies in Mathematics},
	address={Providence, RI}
}

@article{degli-esposti93,
	author={{Degli Esposti}, Mirko},
	title={{Quantization of the orientation preserving automorphisms of the torus}},
year={1993},
volume={58},
number={3},
pages={323--341},
journal={Annales de l'Institut Henri Poincaré}
}

@article{zelditch87,
	author={Zelditch, Steven},
	title={{Uniform distribution of eigenfunctions on compact hyperbolic surfaces}},
		year={1987},
		volume={55},
		number={4},
		journal={Duke Mathematical Journal},
		pages={919--941}
}

@article{ruelle86,
author={Ruelle, David},
title={Resonances of Chaotic Dynamical Systems},
year={1986},
journal={Physical Review Letters},
volume={56},
number={5},
pages={405--407}
}

@article{brini01,
author={Brini, Francesca and Siboni, Stefano},
title = {{Estimates of correlation decay in auto/endomorphisms of the $n$-Torus}},
journal = {Computers \& Mathematics with Applications},
volume = {42},
number = {6},
pages = {941--951},
year = {2001}
}

@article{bouzouina96,
	author={Bouzouina, Abdelkader and de Bièvre, Stephan},
	title={{Equipartition of the eigenfunctions of quantized ergodic maps on the torus}},
	journal={Communications in Mathematical Physics},
	year={1996},
	volume={178},
	number={1},
	pages={83--105}
}

@article{faure03,
	author={Faure, Frédéric and Nonnenmacher, Stéphane and de Bièvre, Stephan},
	title={{Scarred eigenstates for quantum cat maps of minimal periods}},
	year={2003},
	journal={Communications in Mathematical Physics},
	volume={239},
	number={3},
	pages={449--492}
}

\end{document}